\documentclass{eajam}
\usepackage{charter}
\usepackage[charter]{mathdesign}
\usepackage{bm}            

\def\bx{{\bm x}}
\def\by{{\bm y}}
\def\bz{{\bm z}}

\def\be{{\bm e}}            

\def\bT#1{\left\|#1\right\|}              
\def\BT#1{\left\|#1\right\|^2}

\allowdisplaybreaks[1] 

\usepackage{graphicx}
\usepackage{subfigure}
\usepackage{epstopdf}  

\usepackage{threeparttable}

\usepackage{enumerate}

\begin{document}

\markboth{N.-C. Wu and C.-Z. Liu}{The asynchronous LSPIA method}
\title{Asynchronous progressive iterative approximation method for least-squares fitting}
 
\author[N.-C. Wu and C.-Z. Liu]{Nian-Ci Wu \affil{1} and Cheng-Zhi  Liu \affil{2}\comma\corrauth}
\address{\affilnum{1}\ School of Mathematics and Statistics, South-Central Minzu University, Wuhan 430074, China.\\
\affilnum{2}\ School of Mathematics and Finance, Hunan University of Humanities, Science and Technology, Loudi 417000,  China.}
\email{{\tt it-rocket@163.com} (C.-Z. Liu)}
%
\begin{abstract}
For large-scale data fitting, the least-squares progressive-iterative approximation (LSPIA) methods were proposed by Lin et al. (SIAM Journal on Scientific Computing, 2013, 35(6):A3052-A3068) and Deng et al. (Computer-Aided Design, 2014, 47:32-44), where the constant step sizes were used.  In this work, we further accelerate the LSPIA method in the sense of a Chebyshev semi-iterative scheme and present an asynchronous LSPIA (ALSPIA) method to fit data points. The control points in ALSPIA are updated by utilizing an extrapolated variant and an adaptive step size is chosen according to the roots of Chebyshev polynomials. Our convergence analysis reveals that ALSPIA is faster than the original LSPIA  method in both cases of singular and nonsingular least-squares fittings. Numerical examples show that the proposed algorithm is feasible and effective.
\end{abstract}



\maketitle


\section{Introduction}\label{Sec:1}

The least-squares fitting is a classical approach for fitting a blending curve (resp., surfaces) to data points. Given a sequence of points, a blending fitting curve (resp., surfaces) is constructed by minimizing an error criterion that measures the distance from the function to the data points and some smoothness terms. The least-squares progressive-iterative approximation (LSPIA) method, proposed systematically by Deng and Lin \cite{14DL}, is simple and efficient. After an initial blending curve (resp., surfaces) is generated, it iteratively adjusts the control points so that the limit curve (resp., surfaces) can approximate all the data points. Compared to the classical least-squares fitting, such as \cite{03PS}, the LSPIA method admits several useful properties preferred for some applications in geometric modeling, including flexibility and adaptivity.

The fastest convergence rate of LSPIA is given in  \cite[Theorem 3.1]{14DL} when the collocation matrix is of full rank. Lin et al. further studied the convergence of LSPIA for the singular least-squares fitting systems, see \cite[Theorem 2.8]{18LCZ} and \cite[Section 2.3]{13LZ}.
Combining the advantages of generalized B-splines with these of the geometric iterative method, a fresh least-squares method was given by introducing two different kinds of weights in \cite{16ZGT}. The regularized LSPIA method was provided in \cite{18LLGZHS}, which presents a progressive iterative scheme of non-tensor product bivariate spline surfaces. Several well-known accelerations of LSPIA can be found, e.g., in \cite{19EL, 20HW, 21Wang}. Very recently, Rios and J\"{u}ttler proved that LSPIA  is equivalent to a gradient descent method and proposed an extension of LSPIA with parameter correction focusing on the stochastic scheme \cite{22RJ}.

Other extensions of the progressive-iterative approximation (PIA) method for tensor product surface  were presented in \cite{11LZ, 21LLH}. The case of triangular B\'{e}zier surface  was studied by Liu et al. \cite{20LHL}.
To obtain a new control mesh,
Chen et al. \cite{08CLTYC} progressively modified the vertices of a given mesh and proposed a progressive interpolation based on the Catmull-Clark subdivision.
A similar algorithm was developed in \cite{21WLLMD} for the Loop subdivision surface interpolation.
Moreover, Deng and Ma \cite{12DM} proposed a weighted progressive interpolation algorithm for the Loop subdivision surface to improve the convergence rate. Blending the conjugate gradient method \cite{Saad2003} and the PIA method, the interpolation method for Loop and Catmull-Clark subdivision surfaces was given by Hamza and Lin in \cite{22HL}. We refer to the  survey \cite{18LMD} and the references therein for more details.

In this work, we present an accelerated LSPIA method to fit data points by using B-spline basis.
We update the control points by introducing an adaptive step size, rather than a constant one in \cite{14DL}. We call it the asynchronous LSPIA (ALSPIA) method. ALSPIA is deduced in the sense of Chebyshev semi-iterative methods, as stated in Section \ref{Sec:2}. Our convergence analysis reveals that ALSPIA is faster than the original LSPIA when the step size is chosen based on the roots of Chebyshev polynomials. These theoretical results are derived in Section \ref{Sec:3}.
In Section \ref{Sec:4}, we show some numerical experiments which verify our theoretical analysis and demonstrate that, in comparison
with LSPIA \cite{14DL,18LCZ,13LZ}, faster convergence has been obtained by our method.
Finally, we end this work with some conclusions in Section \ref{Sec:5}.

\section{The ALSPIA method}\label{Sec:2}
In this section, we present the ALSPIA method  for curve and surface fittings by using  blending bases.   The process is elaborated on below.

\subsection{The case of curves}
Let $\left\{\mu_i(x)\right\}_{i=0}^n$ be a blending basis sequence, i.e., these functions are nonnegative and satisfy $\sum_{i=0}^n\mu_i(x)=1$.
 Given a set of point set $\left\{\bm q_j\right\}_{j=0}^m \subseteq \mathbb{R}^{2}\ \rm{or}\ \mathbb{R}^{3}$ ($m \geq n$) to be fitted,  each $\bm q_j$ being associated with a parameter $x_j$ for $j\in [m]$, where $[\ell]:=\{0,1,2,\cdots,\ell\}$ for an integer $\ell\geq 1$, and some points ${\bm p}_{i}^{(0)}$ for $i\in[n]$ as the initial control points, we start with an initial curve
\begin{align*}
\mathcal{C}^{(0)}(x)=\sum_{i=0}^n \mu_i(x) {\bm p}_{i}^{(0)},~x\in [x_0,~x_m]
\end{align*}
and compute
\begin{align*}
\left\{ \begin{array}{l}
{\bm p}_{i}^{(1)}  = {\bm p}_{i}^{(0)} + {\bm \delta}_{i}^{(0)},~i\in[n], \vspace{1ex}\\
{\bm \delta}_i^{(0)} = \omega_{0} \sum_{j=0}^m \mu_i(x_j){\bm r}_{j}^{(0)},~i\in[n],\vspace{1ex}\\
{\bm r}_{j}^{(0)}    = {\bm q}_{j} - \mathcal{C}^{(0)}(x_j),~j\in [m],
\end{array}\right .
\end{align*}
where ${\bm p}_{i}^{(1)}$ is the $i$th new control point, ${\bm \delta}_i^{(0)} $ is the $i$th  adjusting vector, ${\bm r}_{j}^{(0)}$ is the $j$th difference vector, and $\omega_{k}$ is an adaptive  step size.

Sequentially, assume that we have obtained the $k$th curve $\mathcal{C}^{(k)}(x)$ and let
\begin{align}\label{eq:ALSPIA-error+iteCurve}
\left\{ \begin{array}{l}
{\bm p}_{i}^{(k+1)}  = {\bm p}_{i}^{(k)} + {\bm \delta}_{i}^{(k)},~i\in[n],\vspace{1ex}\\
{\bm \delta}_i^{(k)} = \omega_{k} \sum_{j=0}^m \mu_i(x_j){\bm r}_{j}^{(k)},~i\in[n],\vspace{1ex}\\
{\bm r}_{j}^{(k)}    = {\bm q}_{j} - \mathcal{C}^{(k)}(x_j),~j\in [m].
\end{array}\right .
\end{align}
The next curve is generated by
\begin{align*}
\mathcal{C}^{(k+1)}(x) =\sum_{i=0}^n  \mu_i(x) {\bm p}_{i}^{(k+1)}.
\end{align*}

In this way, we get a curve sequence $\left\{ \mathcal{C}^{(k)}(x) \right\}_{k=0}^{\infty}$. The parameter sequence $\left\{\omega_k \right\}_{k=0}^{\infty}$ is introduced to accelerate convergence. The geometric interpretation of ALSPIA is intuitively shown
in Figure \ref{fig:ALSPIA-plot}.

 \begin{figure}[!h]
	\centering
	\includegraphics[width=4in,height=3in]{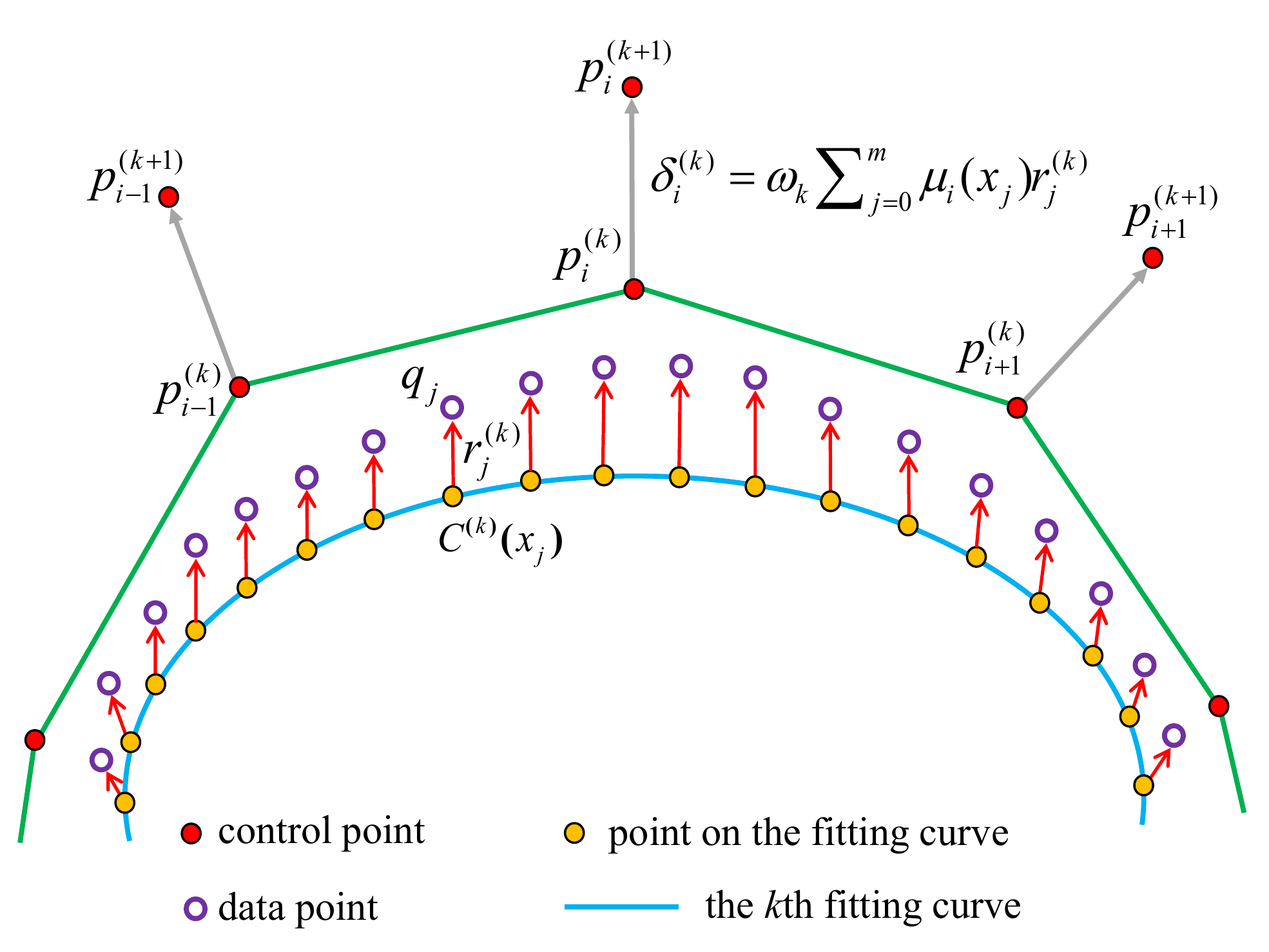}
	\caption{ Geometric interpretation of ALSPIA for curve  fitting. }
	\label{fig:ALSPIA-plot}
\end{figure}

\subsection{The case of surfaces}
 We now turn to the case of surfaces.  Let $\left\{\bm Q_{hl}\right\}_{h,l=0}^{m,p} \subseteq \mathbb{R}^{3}$ ($m, p \geq n$) be a set of point set  to be fitted, each $\bm Q_{hl}$ being associated with a parameter pair $(x_h, y_l)$ for $h\in [m]$ and $l\in [p]$ and ${\bm P}_{ij}^{(0)}$ for $i$, $j\in[n]$ be the initial control points. We construct an initial surface
\begin{align*}
\mathcal{S}^{(0)}(x,y)=\sum_{i=0}^n \sum_{j=0}^n \mu_i(x) \mu_j(y) {\bm P}_{ij}^{(0)},~x\in [x_0,~x_m],~y\in [y_0,~y_p]
\end{align*}
and compute
\begin{align*}
\left\{ \begin{array}{l}
{\bm P}_{ij}^{(1)} = {\bm P}_{ij}^{(0)} +  {\bm \Delta}_{ij}^{(0)},~i,~j\in [n],\vspace{1ex}\\
{\bm \Delta}_{ij}^{(0)}= \omega_0 \sum_{h=0}^m \sum_{l=0}^p \mu_i(x_h) \mu_j(y_l){\bm R}_{hl}^{(0)},~i,~j\in [n],\vspace{1ex}\\
{\bm R}_{hl}^{(0)} = {\bm Q}_{hl} - \mathcal{S}^{(0)}(x_h,y_l),~h\in [m],~l\in [p],
\end{array}\right .
\end{align*}
where ${\bm P}_{ij}^{(1)}$ is the $(i,j)$th new control point, ${\bm \Delta}_{ij}^{(0)}$ is the $(i,j)$th  adjusting vector, ${\bm R}_{hl}^{(0)}$ is the $(h,l)$th difference vector, and $\omega_0$ is a step size.

In the preparation, assume that we have obtained the $k$th surface $\mathcal{S}^{(k)}(x,y)$ and computed
\begin{align}\label{eq:ALSPIA-error+iteSurface}
\left\{ \begin{array}{l}
{\bm P}_{ij}^{(k+1)} = {\bm P}_{ij}^{(k)} +  {\bm \Delta}_{ij}^{(k)},~i,~j\in [n], \vspace{1ex}\\
{\bm \Delta}_{ij}^{(k)}= \omega_k \sum_{h=0}^m \sum_{l=0}^p \mu_i(x_h) \mu_j(y_l){\bm R}_{hl}^{(k)},~i,~j\in [n], \vspace{1ex} \\
{\bm R}_{hl}^{(k)} = {\bm Q}_{hl} - \mathcal{S}^{(k)}(x_h,y_l),~h\in [m],~l\in [p],
\end{array}\right .
\end{align}
the next surface is generated by
\begin{align*}
\mathcal{S}^{(k+1)}(x,y)=\sum_{i=0}^n \sum_{j=0}^n \mu_i(x) \mu_j(y) {\bm P}_{ij}^{(k+1)}.
\end{align*}

Finally, we get a surface sequence $\left\{ \mathcal{S}^{(k)}(x,y) \right\}_{k=0}^{\infty}$, where the parameter sequence $\left\{\omega_k \right\}_{k=0}^{\infty}$ is used to speed up convergence.

\section{Convergence analyses of ALSPIA}\label{Sec:3}
In this section, we utilize matrix theory to analyze the convergence of the ALSPIA method.
\subsection{The ALSPIA method for curve fitting} Let the data points and control points be arranged respectively into
\begin{align*}
{\bm q}=\left[{\bm q}_{0}~{\bm q}_{1}~\cdots ~{\bm q}_{m}\right]^T
~{\rm and}~
{\bm p}^{(k)}=\left[ {\bm p}_{0}^{(k)}~{\bm p}_{1}^{(k)}~\cdots~{\bm p}_{n}^{(k)}\right]^T
\end{align*}
for $k=0,1,2,\cdots$. Define the collocation matrix of a system $\left(\mu_0(x)~\mu_1(x)~\cdots~\mu_n(x)\right)$ at the real increasing sequence $\left\{x_j \right\}_{j=0}^m$ as
\begin{align*}
{\bm A} = \begin{bmatrix} \mu_i(x_j)  \end{bmatrix}_{j=0,i=0}^{m,n} =
\begin{bmatrix}
\mu_0(x_0)&\mu_1(x_0)&\cdots&\mu_n(x_0)\\
\mu_0(x_1)&\mu_1(x_1)&\cdots&\mu_n(x_1)\\
\vdots&\vdots&\ddots&\vdots\\
\mu_0(x_m)&\mu_1(x_m)&\cdots&\mu_n(x_m)\\
\end{bmatrix}.
\end{align*}
Then, the ALSPIA iterative process for curve fitting can be condensed into matrix form
\begin{align}\label{ALSPIA+Iteration}
{\bm p}^{(k+1)} = {\bm p}^{(k)} + \omega_k {\bm A}^T \left({\bm q} - {\bm A} {\bm p}^{(k)} \right).
\end{align}

\begin{remark}\label{rem_singular}
 Let ${\bm \Lambda}$ be a diagonal matrix. Both LSPIAs with singular or nonsingular and ALSPIA admit the matrix form,
 \begin{align}\label{LSPIA+Iteration}
{\bm p}^{(k+1)} = {\bm p}^{(k)} + {\bm \Lambda} {\bm A}^T \left({\bm q} - {\bm A} {\bm p}^{(k)} \right)
\end{align}
for $k=0,1,\cdots$. This form gives us many flexibilities to update the control points and yield various specific instantiations. In particular, if we take ${\bm \Lambda} = \omega_k {\bm I}_{n+1}$ and ${\bm \Lambda} = \omega {\bm I}_{n+1}$, where ${\bm I}_{n+1}$ is the identity matrix with size $n+1$ and $\omega$ is a constant step size, ALSPIA and nonsingular LSPIA \cite{14DL} are obtained, respectively. In a similar way, the singular LSPIA method is recovered by setting ${\bm \Lambda}_{ii} = 1/\sum_{j\in U_i}\mu_i(x_j)$, where $U_i$ is a index set of the $i$th data point group for  $i\in[n]$, see \cite{18LCZ}.
\end{remark}
%
%
%
%

By the linear algebra theory, ${\bm p}^{\ast}$ is a least-squares solution of ${\bm A}{\bm p} = {\bm q}$ if and only if ${\bm A}^T{\bm A}{\bm p}^{\ast} = {\bm A}^T{\bm q}$. Let the symbol $\dag$ denote the Moore-Penrose  pseudoinverse \cite{97Demmel, Golub2013, Saad2003}. By the pseudoinverse identity ${\bm A}^T{\bm A}{\bm A}^{\dag}{\bm q}={\bm A}^T{\bm q}$, we know that ${\bm p}^{\ast}={\bm A}^{\dag}{\bm q}$ is one of the least-squares solution. When the normal system has multiple solutions, ${\bm p}^{\ast}$ has the least Euclidean-norm.

\textbf{Case 1:} ${\bm A}^T{\bm A}$ is singular. Let ${\be}^{(k)} = {\bm p}^{(k)} - {\bm p}^{\ast}$ for $k=0,1,2,\cdots$. It follows that
  \begin{align*}
{\bm A}^T{\bm A}{\be}^{(k+1)}
& = {\bm A}^T{\bm A}{\bm p}^{(k+1)} - {\bm A}^T{\bm A}{\bm A}^{\dag}{\bm q}\\
& = ({\bm A}^T{\bm A})\left({\bm I}_{n+1} -  \omega_{k}{\bm A}^T{\bm A} \right)\left({\bm p}^{(k)} - {\bm A}^{\dag}{\bm q}\right)\\
& = \left({\bm I}_{n+1} -  \omega_{k}{\bm A}^T{\bm A} \right)({\bm A}^T{\bm A}){\be}^{(k)}.
\end{align*}
 Iterating this recurrence,  we have
  \begin{align*}
{\bm A}^T{\bm A}{\be}^{(k+1)}
&= \prod_{\ell=0}^{k}\left({\bm I}_{n+1} - \omega_\ell {\bm A}^T{\bm A}\right)({\bm A}^T{\bm A}){\be}^{(0)}\\
&= ({\bm A}^T{\bm A})\prod_{\ell=0}^{k}\left({\bm I}_{n+1} - \omega_\ell {\bm A}^T{\bm A}\right){\be}^{(0)}\\
&:= \widetilde{P}_{k+2}({\bm A}^T{\bm A}) {\be}^{(0)}.
\end{align*}

\textbf{Case 2:} ${\bm A}^T{\bm A}$ is nonsingular. In this case, the solution ${\bm p}^{\ast} = ({\bm A}^T{\bm A})^{-1}{\bm A}^T{\bm q}$, it yields that
\begin{align*}
{\be}^{(k+1)}
& = {\bm p}^{(k+1)}- ({\bm A}^T{\bm A})^{-1}{\bm A}^T{\bm q}\\
& = {\bm p}^{(k)}+ \omega_k \left( ({\bm A}^T{\bm A}) ({\bm A}^T{\bm A})^{-1}{\bm A}^T {\bm q} - {\bm A}^T{\bm A} {\bm p}^{(k)} \right) - ({\bm A}^T{\bm A})^{-1}{\bm A}^T{\bm q} \\
& = \left({\bm I}_{n+1} -  \omega_k{\bm A}^T{\bm A} \right){\be}^{(k)}.
\end{align*}
This recurrence indicates that
\begin{align*}
{\be}^{(k+1)} = \prod_{\ell=0}^{k}\left({\bm I}_{n+1} - \omega_\ell {\bm A}^T{\bm A}\right){\be}^{(0)}:= \widehat{P}_{k+1}({\bm A}^T{\bm A}){\be}^{(0)}.
\end{align*}

General convergence results seen in  \cite{97Demmel, Golub2013, Saad2003} imply that if the spectral radius of the iteration matrix is less than unity, or equivalently,
\begin{align*}
  \lim_{k\rightarrow \infty} \rho\left(\widetilde{P}_{k}({\bm A}^T{\bm A})\right) = 0
  ~~ {\rm or} ~~
  \lim_{k\rightarrow \infty} \rho\left(\widehat{P}_{k}({\bm A}^T{\bm A})\right) = 0,
\end{align*}
where $\rho({\bm D})$ is the spectral radius of any squared matrix ${\bm D}$,  then  the iteration \eqref{ALSPIA+Iteration} converges toward the required solution and the limit curve, generated by the ALSPIA method, is the least-squares fitting result to the initial data. Next, we will discuss the convergence of the ALSPIA method divided into ${\bm A}^T{\bm A}$ being singular and nonsingular cases, where the adaptive step size is chosen based on the roots of Chebyshev polynomials (see Appendix for a brief review of the main properties of Chebyshev polynomials).

\subsubsection{Case 1: ${\bm A}^T{\bm A}$ is singular} In this case, we get the sub-linear convergence rate for ALSPIA.

\begin{theorem}\label{thm:ALSPIA-Rate1}
Let $\left\{\mu_i(x)\right\}_{i=0}^n$ be a blending basis and ${\bm A}$ be the corresponding collocation matrix on the real increasing sequence $\left\{x_j \right\}_{j=0}^m$. Suppose that the step size sequence $\left\{\omega_\ell\right\}_{\ell=0}^{k-1}$
 depends on the roots of the Chebyshev polynomial given by
 \begin{align}\label{Thm:Result1+singular}
   \omega_\ell =
   \frac{1-\cos\left(\frac{2k+1}{2(k+1)}\pi \right)}
   {v \left(\cos\left(\frac{2\ell+1}{2(k+1)}\pi \right) - \cos\left(\frac{2k+1}{2(k+1)}\pi \right) \right)},~\ell\in[k-1],
 \end{align}
for a fixed number of iterations $k$, where $v$ is the largest eigenvalue of $\bar{{\bm A}}={\bm A}^T{\bm A}$. The limit of sequence $\left\{ \mathcal{C}^{(k)}(x) \right\}_{k=0}^{\infty}$, generated by the  ALSPIA method, is the least-squares curve of the initial data $\left\{{\bm q}_j\right\}_{j=0}^m$ even though ${\bm A}$ is rank deficient. In such a case, we have the following sub-linear convergence rate
\begin{align}\label{Thm:Result2+singular}
 \rho\left(\widetilde{P}_{k+1}(\bar{{\bm A}}) \right) \leq \frac{v \pi}{2(k+1)^2},
\end{align}
where $\widetilde{P}_{k+1}(\bar{{\bm A}})=\bar{{\bm A}} \prod_{\ell=0}^{k-1}\left({\bm I}_{n+1} - \omega_\ell \bar{{\bm A}}\right)$.
\end{theorem}

\begin{proof}
Since $\bar{{\bm A}}$ is singular, the spectrum of $\bar{{\bm A}}$ satisfies that
 \begin{align*}
0\leq \lambda_{i}(\bar{{\bm A}}) \leq\lambda_{\max}(\bar{{\bm A}})=v<\infty,
\end{align*}
where $\lambda_{i}(\bar{{\bm A}})$ be the $i$th eigenvalue of $\bar{{\bm A}}$ for $i=1,2,\cdots,n+1$. It follows that
\begin{align*}
  \rho\left(\widetilde{P}_{k+1}(\bar{{\bm A}})\right)= \max_{i=1,2,\cdots,n+1}\left| \widetilde{P}_{k+1}(\lambda_{i})  \right|
  \leq \max_{\lambda\in [0,~v]}\left| \widetilde{P}_{k+1}(\lambda)  \right|,
\end{align*}
where
\begin{align}\label{eq:tildePk}
\widetilde{P}_{k+1}(\lambda) = \lambda \prod_{\ell=0}^{k-1}\left(1 - \omega_\ell \lambda \right)
\end{align}
is the polynomial of degree $k + 1$. Then, we aim at selecting the weights $\omega_\ell$ for $\ell\in[k-1]$ such that $\widetilde{P}_{k+1}(\lambda)$ has the least deviation from $0$ on the interval $[0,~v]$ and satisfying
\begin{align*}
 \widetilde{P}_{k+1}(0)=0 ~~ {\rm and} ~~ \widetilde{P}_{k+1}^{'}(0)=1,
\end{align*}
which can be given in terms of a Chebyshev polynomial.

Let $P_{k+1}(x)$ be the first kind of Chebyshev polynomial of degree $k+1$ whose the closest root to $1$ is
\begin{align*}
x_{k+1} = \cos\left(\frac{2k+1}{2(k+1)} \pi \right) = \cos\left(\pi - \frac{1}{2(k+1)} \pi \right) < 0.
\end{align*}
Define the polynomial
\begin{align*}
\widetilde{P}_{k+1}(\lambda) =
\frac{v}{1-x_{k+1}}
\frac{P_{k+1}\left(x_{k+1} + \frac{1-x_{k+1}}{v}\lambda \right)}{P_{k+1}^{'}(x_{k+1})}.
\end{align*}
It is easy to check that
\begin{align*}
\widetilde{P}_{k+1}(0) =
\frac{v}{1-x_{k+1}}
\frac{P_{k+1}\left(x_{k+1} \right)}{P_{k+1}^{'}(x_{k+1})} = 0
\end{align*}
and
\begin{align*}
\widetilde{P}_{k+1}^{'}(0) =
\frac{v}{1-x_{k+1}} \frac{1-x_{k+1}}{v}
\frac{P_{k+1}^{'}(x_{k+1})}{P_{k+1}^{'}(x_{k+1})} = 1.
\end{align*}
According to this choice of $\widetilde{P}_k(\lambda)$, it follows that
\begin{align*}
\max_{\lambda\in [0,~v]}\left| \widetilde{P}_{k+1}(\lambda)  \right|
 = \max_{\lambda\in [0,~v]}\left| \frac{v}{1-x_{k+1}}
\frac{P_{k+1}\left(x_{k+1} + \frac{1-x_{k+1}}{v}\lambda \right)}{P_{k+1}^{'}(x_{k+1})}  \right|
 \leq \frac{v}{\left| P_{k+1}^{'}(x_{k+1}) \right|},
\end{align*}
where the inequality is from Lemma \ref{lemma:Chebyshev+Demmel}.  Moreover, from the property of Chebyshev polynomial, we know that, for any $\theta$,
\begin{align*}
P_{k+1}(\cos(\theta)) = \cos((k+1)\theta).
\end{align*}
Differentiating on both sides yields that
\begin{align*}
\sin(\theta) P_{k+1}^{'}(\cos(\theta)) = (k+1) \sin((k+1)\theta).
\end{align*}
When $\theta = \pi - \pi/(2k+2)$, we have
 \begin{align*}
 \left| P_{k+1}^{'}(x_{k+1}) \right|
 &= \frac{(k+1)|\sin((k+1)\pi-\pi/2)|}{|\sin(\pi-\pi/(2k+2))|}\\
 &= \frac{k+1}{|\sin(\pi-\pi/(2k+2))|}\\
 &= \frac{2(k+1)^2}{\pi},
 \end{align*}
where the last line is from the fact that $\sin(\pi-\pi/(2k+2))\sim \pi/(2k+2)$  when $k$ is sufficiently large. Then we can straightforwardly obtain the estimate
\begin{align*}
\rho\left( \widetilde{P}_{k+1}(\bar{{\bm A}}) \right)
\leq \frac{v}{\left| P_{k+1}^{'}(x_{k+1}) \right|}
\leq \frac{v \pi}{2(k+1)^2}.
\end{align*}
Formula \eqref{eq:tildePk} shows that the zeros of polynomial $\widetilde{P}_{k+1}(\lambda)$ are $\lambda = 1/\omega_\ell$,  for $\ell\in[k-1]$. Hence, the optimal step sizes (for the approximation problem, simplified to on interval) are given by the inverse roots of $\widetilde{P}_{k+1}(\lambda)$ \cite[Chapter 1]{2014Maxim}. That is,
\begin{align*}
x_{k+1} + \frac{1-x_{k+1}}{v}\omega_{\ell}^{-1} = \cos\left(\frac{2\ell+1}{2(k+1)}\pi\right).
\end{align*}
Therefore, we obtain the statement in \eqref{Thm:Result1+singular}. \hfill
\end{proof}

\subsubsection{Case 2: ${\bm A}^T{\bm A}$ is nonsingular} In this case, we get the following linear convergence rate for ALSPIA.

\begin{theorem}\label{thm:ALSPIA-Rate2}
Let $\left\{\mu_i(x)\right\}_{i=0}^n$ be a blending basis and ${\bm A}$ be the corresponding collocation matrix on the real increasing sequence $\left\{x_j \right\}_{j=0}^m$. Suppose that the step size sequence $\left\{\omega_\ell\right\}_{\ell=0}^{k-1}$
 depends on the roots of the Chebyshev polynomial given by
 \begin{align}\label{Thm:Result1+nonsingular}
  \omega_\ell = 2\left( (v + u) + (v - u) \cos\left(\frac{2\ell+1}{2k}\pi\right)\right)^{-1},~\ell\in[k-1],
 \end{align}
for a fixed number of iterations $k$, where $v$ and $u$ are the largest and the smallest eigenvalues of $\bar{{\bm A}}$, respectively. The limit of sequence $\left\{ \mathcal{C}^{(k)}(x) \right\}_{k=0}^{\infty}$, generated by the  ALSPIA method, is the least-squares curve of the initial data $\left\{{\bm q}_j\right\}_{j=0}^m$ when ${\bm A}$ is of full rank. In such a case, we have the following linear convergence rate
\begin{align}\label{Thm:Result2+nonsingular}
 \rho\left( \widehat{P}_k(\bar{{\bm A}}) \right) \leq  2\left(\frac{\sqrt{v}-\sqrt{u}}{\sqrt{v}+\sqrt{u}}\right)^k,
\end{align}
where $\widehat{P}_k(\bar{{\bm A}})= \prod_{\ell=0}^{k-1}\left({\bm I}_{n+1} - \omega_\ell \bar{{\bm A}}\right)$.
\end{theorem}

\begin{proof}
Assume that the spectrum of $\bar{{\bm A}}$ satisfies that
 \begin{align*}
0<u=\lambda_{\min}(\bar{{\bm A}}) \leq \lambda_{i}(\bar{{\bm A}}) \leq\lambda_{\max}(\bar{{\bm A}})=v<\infty,~i=1,2,\cdots,n+1.
\end{align*}
Then, we have the following bound
\begin{align*}
  \rho\left( \widehat{P}_k(\bar{{\bm A}}) \right)
  = \max_{i=1,2,\cdots,n+1}\left| \widehat{P}_k(\lambda_{i})  \right|
  \leq \max_{\lambda\in [u,~v]}\left| \widehat{P}_k(\lambda)  \right|.
\end{align*}
In such way, we can choose the asynchronous sizes $\omega_\ell$ for $\ell\in[k-1]$ such that $\widehat{P}_k(\lambda)$ is the polynomial least deviation from zero on $[u,~v]$, which satisfies that $\widehat{P}_k(0)=1$. From Lemma \ref{lemma:Chebyshev2}, it yields that
\begin{align*}
\max_{\lambda\in [u,~v]}\left| \widehat{P}_k(\lambda)  \right|
=
\max_{\lambda\in [u,~v]}\left|
\frac{P_{k}\left( \frac{2\lambda}{v-u} -  \frac{v+u}{v-u}\right)}{P_{k}\left( -  \frac{v+u}{v-u}\right)}
\right|
=\frac{1}{\left|P_{k}\left( -  \frac{v+u}{v-u}\right)\right|}
=\frac{1}{P_{k}\left(\frac{v+u}{v-u}\right)},
\end{align*}
where $P_{k}(x)$ is the $k$th Chebyshev polynomial. Let $s= (v+u)/(v-u) = (1+t)/(1-t)>1$ with $t=u/v\in(0,~1]$. We have
\begin{align*}
  s^2 - 1 = \frac{(1+t)^2 - (1-t)^2}{(1-t)^2} = \frac{4t}{(1-t)^2}.
\end{align*}
An elementary computation shows that
\begin{align*}
 s + \sqrt{s^2 - 1} = \frac{1+t+2\sqrt{t}}{1-t} = \frac{(1+\sqrt{t})^2}{1-t} = \frac{1+\sqrt{t}}{1-\sqrt{t}}
\end{align*}
and
\begin{align*}
 s - \sqrt{s^2 - 1} = \frac{1+t-2\sqrt{t}}{1-t} = \frac{(1-\sqrt{t})^2}{1-t} = \frac{1-\sqrt{t}}{1+\sqrt{t}}.
\end{align*}
From Lemma \ref{lemma:Chebyshev+Demmel}, we know that
\begin{align*}
 P_{k}\left(\frac{v+u}{v-u}\right)
 & = P_{k}\left(s \right) = \frac{1}{2}\left[ (s+\sqrt{s^2-1})^k + (s+\sqrt{s^2-1})^{-k} \right]\\
 & = \frac{1}{2} \left[ \left( \frac{1+\sqrt{t}}{1-\sqrt{t}}  \right)^k + \left( \frac{1-\sqrt{t}}{1+\sqrt{t}} \right)^k  \right]\\
 & \geq \frac{1}{2} \left( \frac{1+\sqrt{t}}{1-\sqrt{t}} \right)^k.
\end{align*}
It follows that
\begin{align*}
  \rho\left( \widehat{P}_k(\bar{{\bm A}}) \right)
  \leq \max_{\lambda\in [u,~v]}\left| \widehat{P}_k(\lambda)  \right|
  \leq 2\left(\frac{1-\sqrt{t}}{1+\sqrt{t}}\right)^k
  =  2\left(\frac{\sqrt{v}-\sqrt{u}}{\sqrt{v}+\sqrt{u}}\right)^k.
\end{align*}
Correspondingly, the step sizes $\omega_{\ell}$, for $\ell\in[k-1]$, are chosen as the inverse roots of polynomial $\widetilde{P}_k(\lambda)$. That is,
\begin{align*}
\frac{2\omega_{\ell}^{-1}}{v-u} -  \frac{v+u}{v-u} = \cos\left(\frac{2\ell+1}{2k}\pi\right),
\end{align*}
which yields the expression in formula \eqref{Thm:Result1+nonsingular} immediately. \hfill
\end{proof}

\begin{remark}
Formulas \eqref{Thm:Result1+singular} and \eqref{Thm:Result1+nonsingular} give the sets of the step sizes $\omega_{\ell}$ with $\ell\in[k-1]$ for a fixed $k$. The order, where $\omega_{\ell}$ is used, does not matter. In practice, however, this can be important, and the natural ordering from \eqref{Thm:Result1+singular} or \eqref{Thm:Result1+nonsingular} may lead to an amplification of roundoff error. It is interesting and may be the subject of a new paper. In this work, we do not consider the influence of the order of $\omega_\ell$ and choose it in a cycle fashion. Several re-ordering methods leading to more stable calculations have been proposed in the literature; see, e.g., \cite{71LF}.
\end{remark}

\begin{remark}
Deng and Lin (\cite[Theorem 3.1]{14DL}) have succeeded to establish an upper bound on the step size that guarantees the convergence of LSPIA. When the optimal step size is given by $\omega_{\rm opt} = 2/(v + u)$, LSPIA has the smallest possible convergence rate  $\rho_{\textsc{lspia}}=(v - u)/(v + u)$ for the case of curve fitting. When ${\bm A}$ is of full-column rank, the matrix $\bar{{\bm A}}$ has all positive eigenvalues then $u/v \in (0,1]$ which means that $\rho_{\textsc{lspia}}\in [0,1)$. Then, LSPIA is guaranteed to converge. In addition, we can see that the convergence rate, given by Theorem \ref{thm:ALSPIA-Rate2}, is better than $\rho_{\textsc{lspia}}$ for a fixed $k$. Note also that the convergence rate from Theorem \ref{thm:ALSPIA-Rate2} is the same as that for the conjugate gradient method; see, for example, formula 6.128 in \cite{Saad2003}, and it is optimal for this class of iterative schemes.
\end{remark}

\begin{remark}
The ALSPIA method with Chebyshev-based step size belongs to the class of Chebyshev semi-iterative methods \cite{61GV}, however, which to the best of our knowledge has not been previously used to accelerate the convergence rate of LSPIA \cite{14DL}. Some other accelerations of LSPIA have been proposed, e.g., in \cite{19EL, 20HW}. Take the curve fitting situation as an example. In \cite{19EL}, the Schulz composite iterative procedure   is applied to update the  adjusting vector. In \cite[Equation (3)]{20HW}, by introducing three real weights, the LSPIA with memory method needs two additional storing data of the previous step and three more scalar multiplications at each iteration to calculate the adjusting vector, its optimal convergence rate was presented in \cite[Theorem 6]{20HW}.
\end{remark}

\subsection{The ALSPIA method for surface fitting}  In this case, we put all the coordinates of data point and control point into a row partition, which are arranged as follows.
\begin{align*}
{\bm Q} = \left[{\bm Q}_{00}~\cdots~{\bm Q}_{m0}~{\bm Q}_{01}~\cdots~{\bm Q}_{m1}~ \cdots ~{\bm Q}_{0p}~\cdots~{\bm Q}_{mp}~\right]^T,
\end{align*}
and
\begin{align*}
{\bm P}^{(k)} = \left[{\bm P}_{00}^{(k)}~\cdots~{\bm P}_{n0}^{(k)}~{\bm P}_{01}^{(k)}~\cdots~{\bm P}_{n1}^{(k)}~ \cdots ~{\bm P}_{0n}^{(k)}~\cdots~{\bm P}_{nn}^{(k)}~  \right]^T
\end{align*}
for $k=0,1,2,\cdots$. Define another collocation matrix of a system $\left(\mu_0(y)~\mu_1(y)~\cdots~\mu_n(y)\right)$ at the real increasing sequence $\left\{y_l \right\}_{l=0}^p$ as
\begin{align*}
{\bm B}^T = \begin{bmatrix} \mu_j(y_l)  \end{bmatrix}_{l=0,j=0}^{p,n} =
\begin{bmatrix}
\mu_0(y_0)&\mu_1(y_0)&\cdots&\mu_n(y_0)\\
\mu_0(y_1)&\mu_1(y_1)&\cdots&\mu_n(y_1)\\
\vdots&\vdots&\ddots&\vdots\\
\mu_0(y_p)&\mu_1(y_p)&\cdots&\mu_n(y_p)\\
\end{bmatrix}.
\end{align*}
Observe that formula \eqref{eq:ALSPIA-error+iteSurface} enables us to express the iteration of ALSPIA surface fitting in the vectorized form
\begin{align}\label{ALSPIA+Iteration+surface}
{\bm P}^{(k+1)} = {\bm P}^{(k)} + \omega_k \left( {\bm B} \otimes {\bm A}^T \right) \left( {\bm Q} - \left( {\bm B}^T \otimes {\bm A} \right) {\bm P}^{(k)} \right).
\end{align}
It indicates that the approach, used in the case of curve fitting, carries over well to this case. Then, we can get the following convergence results for ALSPIA surface fitting.

\begin{corollary}\label{corollary:ALSPIA-Rate1}
Let $\left\{\mu_i(x)\right\}_{i=0}^n$ and $\left\{\mu_j(y)\right\}_{j=0}^n$ be two blending bases, and ${\bm A}$ and ${\bm B}^T$ be the corresponding collocation matrices on the real increasing sequences $\left\{x_h \right\}_{h=0}^m$ and $\left\{y_l \right\}_{l=0}^p$, respectively. Suppose that the step size sequence $\left\{\omega_\ell\right\}_{\ell=0}^{k-1}$
 depends on the roots of the Chebyshev polynomial given by
 \begin{align*}
   \omega_\ell =
   \frac{1-\cos\left(\frac{2k+1}{2(k+1)}\pi \right)}
   {\widetilde{v} \left(\cos\left(\frac{2\ell+1}{2(k+1)}\pi \right) - \cos\left(\frac{2k+1}{2(k+1)}\pi \right) \right)},~\ell\in[k-1],
 \end{align*}
for a fixed number of iterations $k$, where $\widetilde{v}$ is the largest eigenvalue of $\bar{\bar{{\bm A}}}=({\bm B}^T \otimes {\bm A})^T({\bm B}^T \otimes {\bm A})$. The limit of sequence $\left\{ \mathcal{S}^{(k)}(x,y) \right\}_{k=0}^{\infty}$, generated by the  ALSPIA method, is the least-squares surface of the initial data $\left\{{\bm Q}_{hl} \right\}_{h=0,l=0}^{m,p}$ even though ${\bm B}^T \otimes {\bm A}$ is rank deficient. In such a case, we have the following sub-linear convergence rate
\begin{align*}
 \rho\left(\widetilde{P}_{k+1}\left(\bar{\bar{{\bm A}}}\right) \right) \leq \frac{\widetilde{v} \pi}{2(k+1)^2},
\end{align*}
where $\widetilde{P}_{k+1}\left(\bar{\bar{{\bm A}}}\right)=\bar{\bar{{\bm A}}} \prod_{\ell=0}^{k-1}\left({\bm I}_{(n+1)^2} - \omega_\ell \bar{\bar{{\bm A}}}\right)$.
\end{corollary}

\begin{corollary}\label{corollary:ALSPIA-Rate2}
Let $\left\{\mu_i(x)\right\}_{i=0}^n$ and $\left\{\mu_j(y)\right\}_{j=0}^n$ be two blending bases, and ${\bm A}$ and ${\bm B}^T$ be the corresponding collocation matrices on the real increasing sequences $\left\{x_h \right\}_{h=0}^m$ and $\left\{y_l \right\}_{l=0}^p$, respectively.  Suppose that the step size sequence $\left\{\omega_\ell\right\}_{\ell=0}^{k-1}$
 depends on the roots of the Chebyshev polynomial given by
 \begin{align*}
  \omega_\ell = 2\left( (\widehat{v} + \widehat{u}) + (\widehat{v} - \widehat{u}) \cos\left(\frac{2\ell+1}{2k}\pi\right)\right)^{-1},~\ell\in[k-1],
 \end{align*}
for a fixed number of iterations $k$, where $\widehat{v}$ and $\widehat{u}$ are the largest and the smallest eigenvalues of $\bar{\bar{{\bm A}}}$, respectively. The limit of sequence $\left\{ \mathcal{S}^{(k)}(x,y) \right\}_{k=0}^{\infty}$, generated by the  ALSPIA method, is the least-squares surface of the initial data $\left\{{\bm Q}_{hl} \right\}_{h=0,l=0}^{m,p}$ when ${\bm B}^T \otimes {\bm A}$ is of full rank. In such a case, we have the following linear convergence rate
\begin{align*}
 \rho\left( \widehat{P}_k\left(\bar{\bar{{\bm A}}}\right) \right) \leq 2\left(\frac{\sqrt{\widehat{v}}-\sqrt{\widehat{u}}}{\sqrt{\widehat{v}}+\sqrt{\widehat{u}}}\right)^k,
\end{align*}
where $\widehat{P}_k\left(\bar{\bar{{\bm A}}}\right)= \prod_{\ell=0}^{k-1}\left({\bm I}_{(n+1)^2} - \omega_\ell \bar{\bar{{\bm A}}}\right)$.
\end{corollary}

\section{Numerical experiments}\label{Sec:4}
In this section,  we compare the performance of our method with LSPIA \cite{14DL,18LCZ,13LZ} in terms of iteration number (denoted as IT), computing time in seconds (denoted as CPU), and relative fitting error respectively defined by
\begin{equation*}
 {\textsc{E}}_k = \frac{\sum_{i=0}^n \BT{ \sum_{j=0}^m \mu_i(x_j){\bm r}^{(k)}_{j}}}{\sum_{i=0}^n \BT{ \sum_{j=0}^m \mu_i(x_j){\bm r}^{(0)}_{j}}}
\end{equation*}
and
\begin{equation*}
{\textsc{E}}_k = \frac{\sum_{i=0}^n \sum_{j=0}^n \BT{ \sum_{h=0}^m \sum_{l=0}^p\mu_i(x_h)\mu_j(y_l){\bm R}^{(k)}_{hl}}}
 {\sum_{i=0}^n \sum_{j=0}^n \BT{ \sum_{h=0}^m \sum_{l=0}^p\mu_i(x_h)\mu_j(y_l){\bm R}^{(0)}_{hl}}}
\end{equation*}
for curve and surface cases, respectively, when $k=0,1,2,\cdots$. We also report two speed-ups of ALSPIA against LSPIA, which are
defined by
\begin{equation*}
  {\textsc{S}}_{\textsc{it}} = \frac{{\rm IT~of~LSPIA}}{{\rm IT~of~ALSPIA}}
 ~ {\rm and} ~
  {\textsc{S}}_{\textsc{cpu}} = \frac{{\rm CPU~of~LSPIA}}{{\rm CPU~of~ALSPIA}}.
\end{equation*}
The experiments are terminated once ${\textsc{E}}_{k}$ is less than $10^{-6}$ or IT exceeds $10^{4}$ and let ${\textsc{E}}_k$ be ${\textsc{E}}_\infty$.

We test the ALSPIA method for six representative examples for the B-spline curves fitting and tensor product surface fitting. These examples are given as follows (available from \verb"http://paulbourke.net/geometry/").

\begin{example}\label{ex:ALSPIA+curve1}
  $m+1$ points sampled uniformly from a blob-shaped curve, whose polar coordinate equation is
${\bm r} = 2 + 4\cos\left(2\theta_1+ \pi/4\right) + \cos\left(3\theta_1 + \pi/4\right)$ with $0 \leq \theta_1 \leq 2\pi$.
\end{example}

\begin{example}\label{ex:ALSPIA+curve2}
  $m+1$ points sampled from a spherical cardioid curve, whose coordinates are given by
$\bx = 2\cos(\theta_1) - \cos(3\theta_1)$, $\by = 2\sin(\theta_1) - \sin(3\theta_1)$, and $\bz = 2\cos(\theta_1/2)$ with $0 \leq \theta_1 \leq 4\pi$.
\end{example}

\begin{example}\label{ex:ALSPIA+curve3}
  $m+1$ points sampled from a function
  \begin{align*}
    f(\theta_1) = (3+\theta_1)^2\sin(10\theta_1)(\cos^2(\theta_1))/(\theta_1+1)^2
  \end{align*}
  with $0 \leq \theta_1 \leq 2\pi$.
\end{example}

\begin{example}\label{ex:ALSPIA+curve4}
  $m+1$ points sampled from a helix curve, whose coordinates are given by $\bx = 10\cos(\theta_1\pi/3)$, $\by = 10\sin(\theta_1\pi/3)$, and $\bz = \theta_1\pi/3$ with $0 \leq \theta_1 \leq 2\pi$.
\end{example}

\begin{example}\label{ex:ALSPIA+surface1}
$(m+1)\times (p+1)$ points sampled uniformly from a Lemnescate surface, whose parametric equation is given by
\begin{align*}
\left\{ \begin{array}{l}
\bx = \sqrt{|\sin(2\theta_1)|}\cos(\theta_1)\cos(\theta_2),\vspace{1ex}\\
\by = \sqrt{|\sin(2\theta_1)|}\sin(\theta_1)\cos(\theta_2),\vspace{1ex}\\
\bz = \bx^2 - \by^2 + 2\bx\by \tan^2(\theta_2)
\end{array}\right .
\end{align*}
with $0 \leq \theta_1,~\theta_2 \leq \pi$.
\end{example}

\begin{example}\label{ex:ALSPIA+surface2}
$(m+1)\times (p+1)$ points sampled uniformly from the peaks function
\begin{align*}
  f(\theta_1,\theta_2) =  3(1-\theta_1)^2e^{-\theta_1^2 - (\theta_2+1)^2} - 10(\theta_1/5 - \theta_1^3 - \theta_2^5)e^{-\theta_1^2 - \theta_2^2}
   - 1/3e^{-(\theta_1+1)^2 - \theta_2^2}
\end{align*}
 with $-3 \leq \theta_1 \leq 3$ and $-4 \leq \theta_2 \leq 4$.
\end{example}

\subsection{Curve fitting}
The implementation details of LSPIA and ALSPIA for curve fitting are arranged as follows.

Given an ordered point set $\left\{{\bm q}_j\right\}_{j=0}^m$ in $\mathbb{R}^2$ or $\mathbb{R}^3$, we assign the parameter sequence $\left\{x_j \right\}_{j=0}^m$ according to the normalized accumulated chord  parameterization  method, i.e.,
\begin{equation*}
 x_0 = 0,~x_j = x_{j-1}+\frac{\bT{{\bm q}_j-{\bm q}_{j-1}}}{\sum_{s=1}^m \bT{{\bm q}_s-{\bm q}_{s-1}}},
~{\rm and}~
 x_m =1
\end{equation*}
for $j=1,2,\cdots,m$. We choose a cubic B-spline basis, which is popular and effective in computer-aided geometric design research; see, e.g., \cite{14DL}. The knot vector is defined by
\begin{equation*}
  {\bm \nu} = \left\{
  0,~0,~0,~0,
  ~\bar{x}_4,~\bar{x}_5,\cdots,\bar{x}_n,
  ~1,~1,~1,~1
  \right\}
\end{equation*}
with $\bar{x}_{j+3}=(1-\alpha)x_{i-1} + \alpha x_i$, $i = \lfloor jd \rfloor,~\alpha=jd-i$, $d = (m+1)/(n-2)$ for $j=1,2,\cdots,n-3$, and the notation $\lfloor \cdot \rfloor$ is the greatest integer function. For more details on the formulations of generating the parameter sequence and knot vector, we respectively refer to equations (9.5) and (9.69) in the book by Piegl and Tiller \cite{97PT}.  The collocation matrix is realized by applying MATLAB built-in function, e.g., $A= spcol ({\bm \mu}, 4, {\bm \nu})$.
  For $i=1,2,\cdots,n-1$, the initial control points are selected by
 \begin{equation*}
   {\bm p}_0^{(0)} = {\bm q}_0,~
   {\bm p}_i^{(0)} = {\bm q}_{\left\lfloor\frac{(m+1)i}{n}\right\rfloor},~{\rm and}~
   {\bm p}_n^{(0)} = {\bm q}_m,
 \end{equation*}
which is also described in the formula (23) of Deng and Lin \cite{14DL}.

As shown in  \cite[Theorem 3.1]{14DL}, the optimal parameter appeared in LSPIA is given by $\omega=\omega_{\rm opt}$. Note that the extreme singular values are computed via MATLAB function, e.g., {\tt svd}. To ensure fairness, we execute all these methods without explicitly forming them.

(1) \textbf{${\bm A}^T{\bm A}$ is nonsingular.} For the data points in Examples \ref{ex:ALSPIA+curve1} and \ref{ex:ALSPIA+curve2}, we list the numerical results, including the  relative fitting errors, the numbers of iteration steps, the CPU times, and the speed-ups of IT and CPU, for the LSPIA and ALSPIA methods in Tables  \ref{tab:IT+CPU-ex1} and \ref{tab:IT+CPU-ex2}. In these two examples, the collection matrices are of full-column rank. We find that the relative fitting errors of the tested methods are analogous while ALSPIA takes  much fewer iteration counts and CPU times than LSPIA  with various $m$ and $n$, which indicates that ALSPIA is more effective than LSPIA. We also observe that the two speed-ups increase along with the number of control points growing larger for a fixed $m$. In Figures \ref{fig:ex1-ALSPIA1+Result} and \ref{fig:ex2-ALSPIA1+Result}, we  respectively draw the initial data points to be fitted, the cubic B-spline fitting curves, and the convergence behaviors of relative fitting error versus CPU time for LSPIA and ALSPIA when $(m,n)=(8000,1000)$. We see that the curves constructed by ALSPIA approximate the given data points accurately and  the relative fitting errors of ALSPIA decay faster than that of LSPIA.

\begin{table}[!htb]
 \normalsize
\caption{The numerical results of ${\textsc{E}}_\infty$, IT, and CPU for LSPIA and ALSPIA in Example  \ref{ex:ALSPIA+curve1} with various $m$ and $n$.}
\centering
\begin{tabular}{ cccccc}
\cline{1-6}
& ($m$, $n$)        & (8000, 1000) & (8000, 2000)& (8000, 3000)& (10000, 3000)\\
\cline{1-6}
LSPIA  &{${\textsc{E}}_\infty$}&$9.11\times 10^{-7}$&$9.87\times 10^{-7}$&$9.88\times 10^{-7}$&$9.82\times 10^{-7}$\\
       &{IT}        &135&277&898&418\\
       &{CPU}       &1.714&6.881&38.829&20.043\\
\cline{1-6}
ALSPIA &{${\textsc{E}}_\infty$}&$9.76\times 10^{-7}$&$6.67\times 10^{-7}$&$6.72\times 10^{-7}$&$6.46\times 10^{-7}$\\
       &{IT}        &10&9&14&12\\
       &{CPU}       &0.115&0.193&0.497&0.515\\
\cline{1-6}
Speed-up & ${\textsc{S}}_{\textsc{it}}$  &\textbf{13.500}&\textbf{30.778}&\textbf{64.143}&\textbf{34.833}\\
         & ${\textsc{S}}_{\textsc{cpu}}$ &\textbf{14.954}&\textbf{35.688}&\textbf{78.062}&\textbf{38.899}\\
\hline
\hline
& ($m$, $n$)        & (15000, 3000) & (15000, 4000)& (15000, 5000)& (20000, 5000)\\
\cline{1-6}
LSPIA  &{${\textsc{E}}_\infty$}&$9.36\times 10^{-7}$&$9.75\times 10^{-7}$&$9.87\times 10^{-7}$&$9.67\times 10^{-7}$\\
       &{IT}        &201&317&573&280\\
       &{CPU}       &14.462&31.715&74.945& 47.871\\
\cline{1-6}
ALSPIA &{${\textsc{E}}_\infty$}&$6.99\times 10^{-7}$&$ 9.45\times 10^{-7}$&$7.13\times 10^{-7}$&$6.55\times 10^{-7}$\\
       &{IT}        &8&10&7&7\\
       &{CPU}       &0.528&0.875&0.736&1.002\\
\cline{1-6}
Speed-up & ${\textsc{S}}_{\textsc{it}}$  &\textbf{25.125}&\textbf{31.700}&\textbf{81.857} &\textbf{40.000} \\
         & ${\textsc{S}}_{\textsc{cpu}}$ &\textbf{27.366}&\textbf{36.246}&\textbf{101.834}&\textbf{47.782} \\
\cline{1-6}
\end{tabular}
\label{tab:IT+CPU-ex1}
\end{table}

\begin{table}[!htb]
 \normalsize
\caption{The numerical results of ${\textsc{E}}_\infty$, IT, and CPU for LSPIA and ALSPIA in Example  \ref{ex:ALSPIA+curve2} with various $m$ and $n$.}
\centering
\begin{tabular}{ cccccc}
\cline{1-6}
& ($m$, $n$)        & (8000, 1000) & (8000, 2000)& (10000, 1000)& (10000, 2000)\\
\cline{1-6}
LSPIA  &{${\textsc{E}}_\infty$}&$9.27\times 10^{-7}$&$8.57\times 10^{-7}$&$8.53\times 10^{-7}$&$9.09\times 10^{-7}$\\
       &{IT}        &79&78&81&78\\
       &{CPU}       &0.901&1.899&1.295&2.577\\
\cline{1-6}
ALSPIA &{${\textsc{E}}_\infty$}&$4.61\times 10^{-7}$&$6.00\times 10^{-7}$&$4.46\times 10^{-7}$&$5.67\times 10^{-7}$\\
       &{IT}        &5&4&5&4\\
       &{CPU}       &0.046&0.075&0.057&0.095\\
\cline{1-6}
Speed-up & ${\textsc{S}}_{\textsc{it}}$  &\textbf{15.800}&\textbf{19.500}&\textbf{16.200}&\textbf{19.500}\\
         & ${\textsc{S}}_{\textsc{cpu}}$ &\textbf{19.545}&\textbf{25.320}&\textbf{22.566}&\textbf{27.165}\\
\hline
\hline
& ($m$, $n$)        & (12000, 1000) & (12000, 2000)& (14000, 1000)& (14000, 2000)\\
\cline{1-6}
LSPIA  &{${\textsc{E}}_\infty$}&$9.62\times 10^{-7}$&$9.49\times 10^{-7}$&$9.38\times 10^{-7}$&$9.41\times 10^{-7}$\\
       &{IT}        &113&165&109&147 \\
       &{CPU}       &2.029&6.359&2.338&6.814\\
\cline{1-6}
ALSPIA &{${\textsc{E}}_\infty$}&$9.95\times 10^{-7}$&$5.48\times 10^{-7}$&$9.82\times 10^{-7}$&$5.32\times 10^{-7}$\\
       &{IT}        &4&4&4&4\\
       &{CPU}       &0.055&0.116&0.064&0.141\\
\cline{1-6}
Speed-up & ${\textsc{S}}_{\textsc{it}}$  &\textbf{28.250}&\textbf{41.250}&\textbf{27.250}&\textbf{36.750}\\
         & ${\textsc{S}}_{\textsc{cpu}}$ &\textbf{36.799}&\textbf{54.649}&\textbf{36.369}&\textbf{48.384}\\
\cline{1-6}
\end{tabular}
\label{tab:IT+CPU-ex2}
\end{table}

\begin{figure}[!htb]
\centering
    \subfigure[Initial data points]{
		\includegraphics[width=0.45\textwidth]{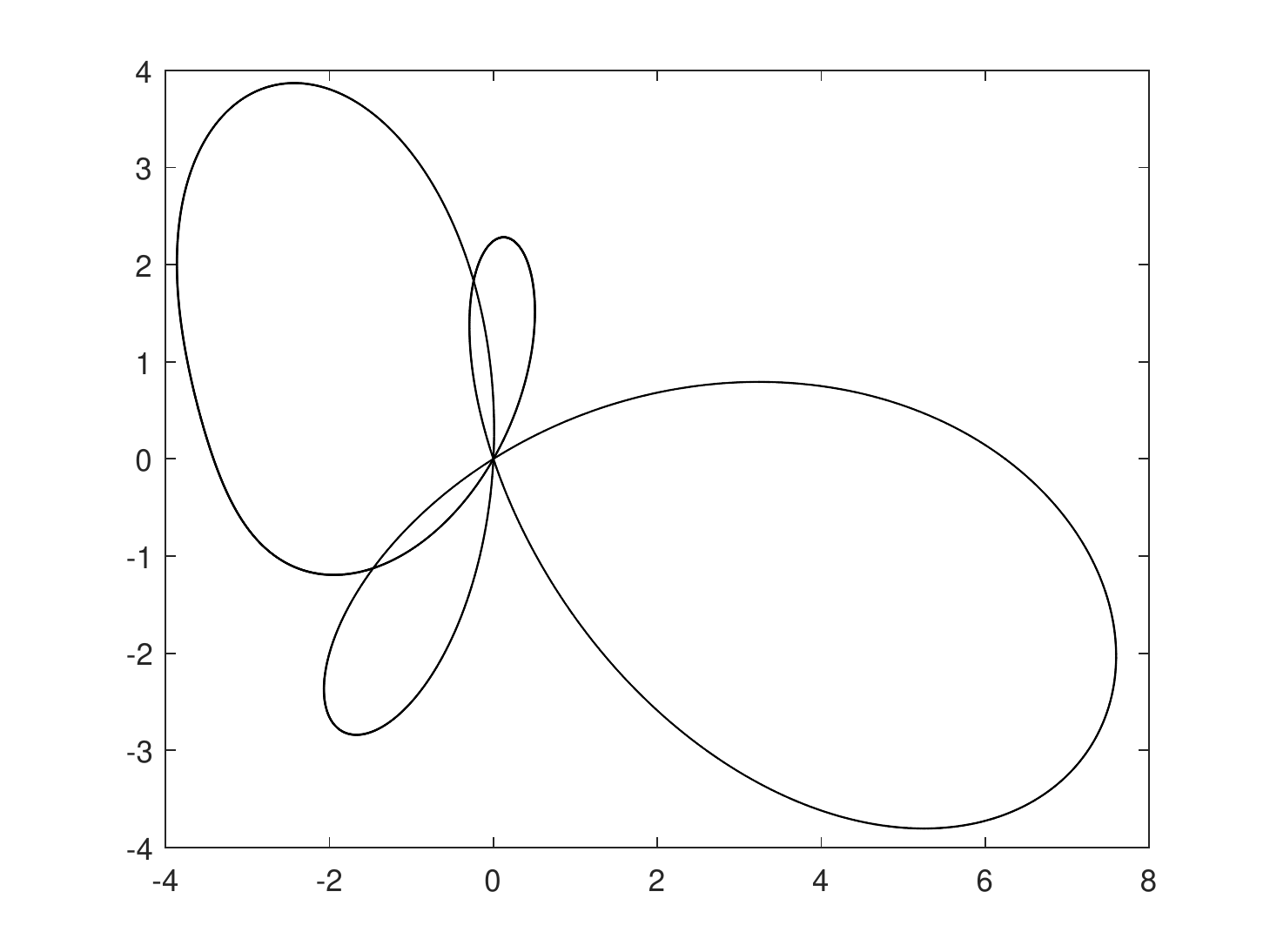}}
	\subfigure[Curve by LSPIA]{
	    \includegraphics[width=0.45\textwidth]{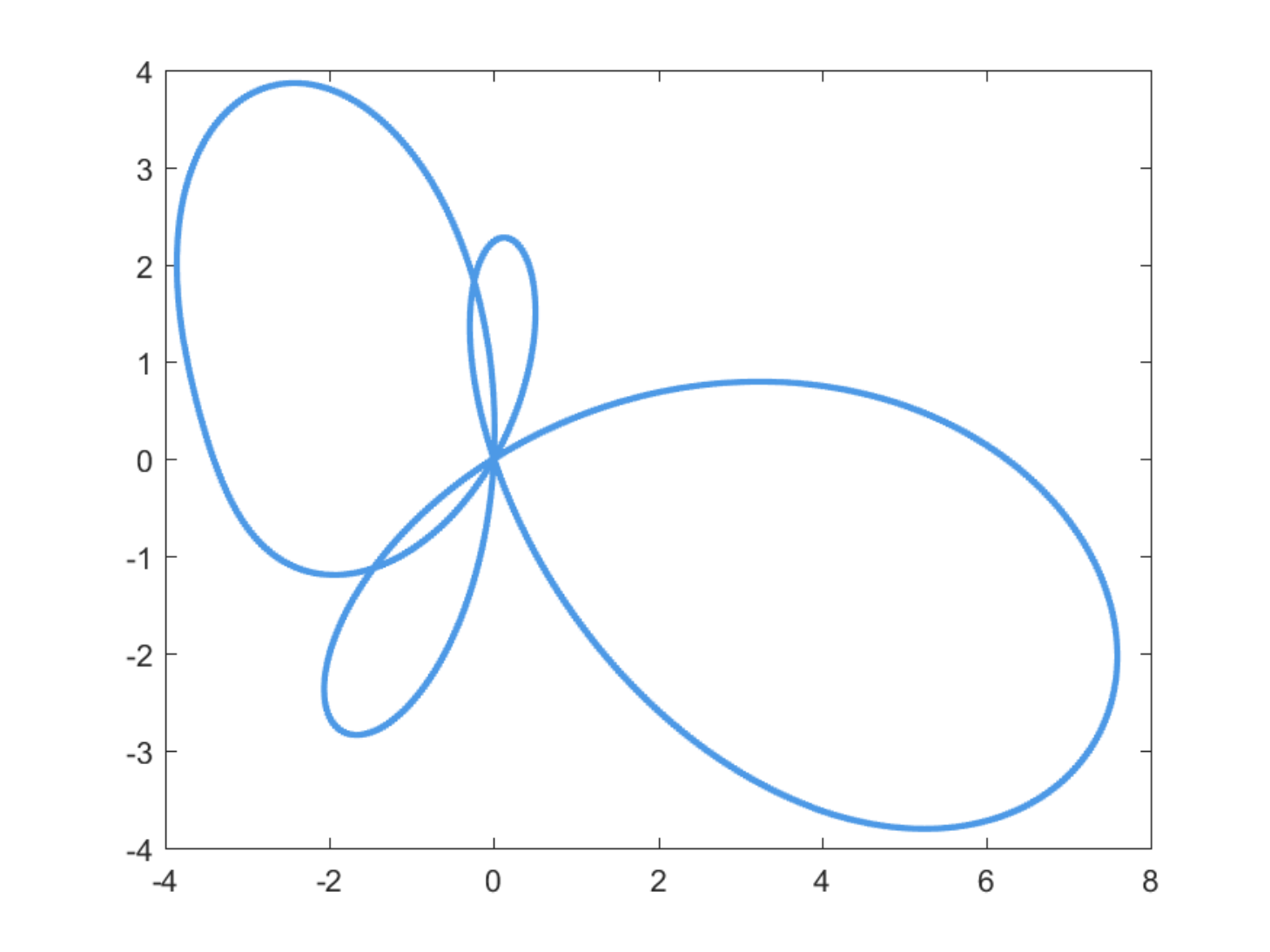}}
	\subfigure[Curve by ALSPIA]{
	    \includegraphics[width=0.45\textwidth]{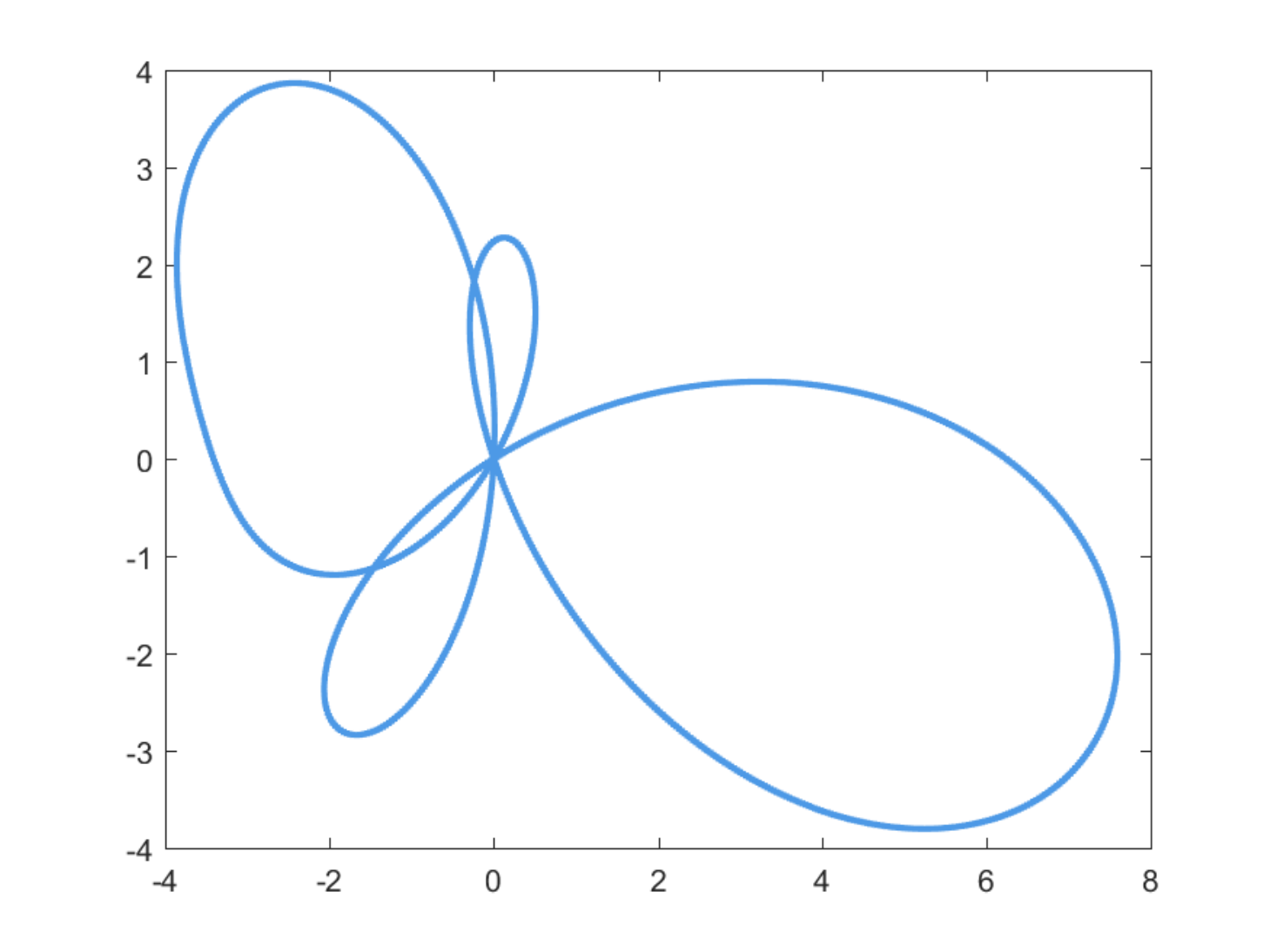}}
    \subfigure[$E_k$ vs CPU]{
	    \includegraphics[width=0.45\textwidth]{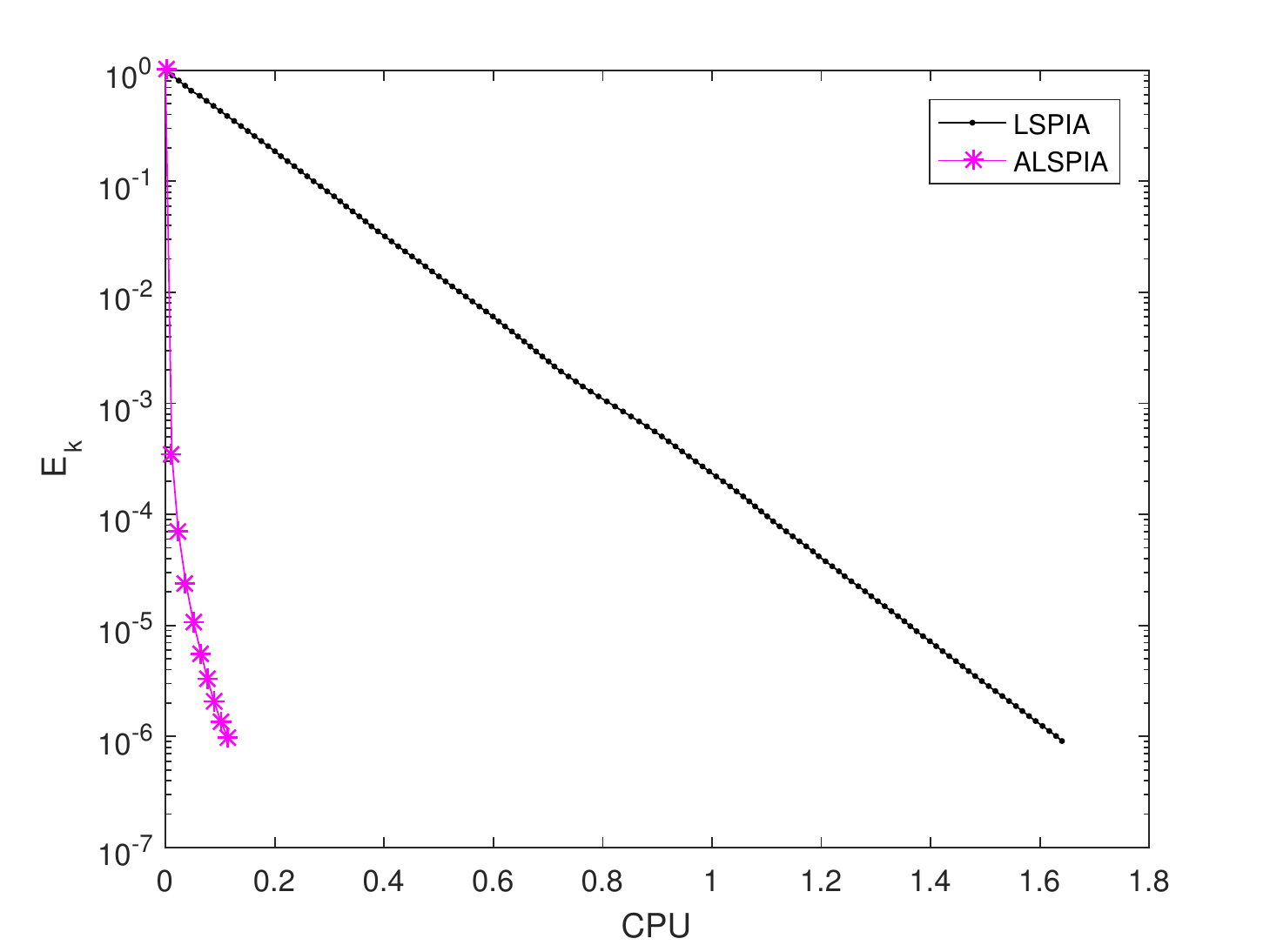}}
\caption{The initial data points, the cubic B-spline fitting curves, and the convergence behaviors of relative fitting error versus CPU time given by LSPIA and ALSPIA with $m=8000$ and $n=1000$ for Example \ref{ex:ALSPIA+curve1}.}
\label{fig:ex1-ALSPIA1+Result}
\end{figure}

\begin{figure}[!htb]
\centering
    \subfigure[Initial data points]{
	    \includegraphics[width=0.45\textwidth]{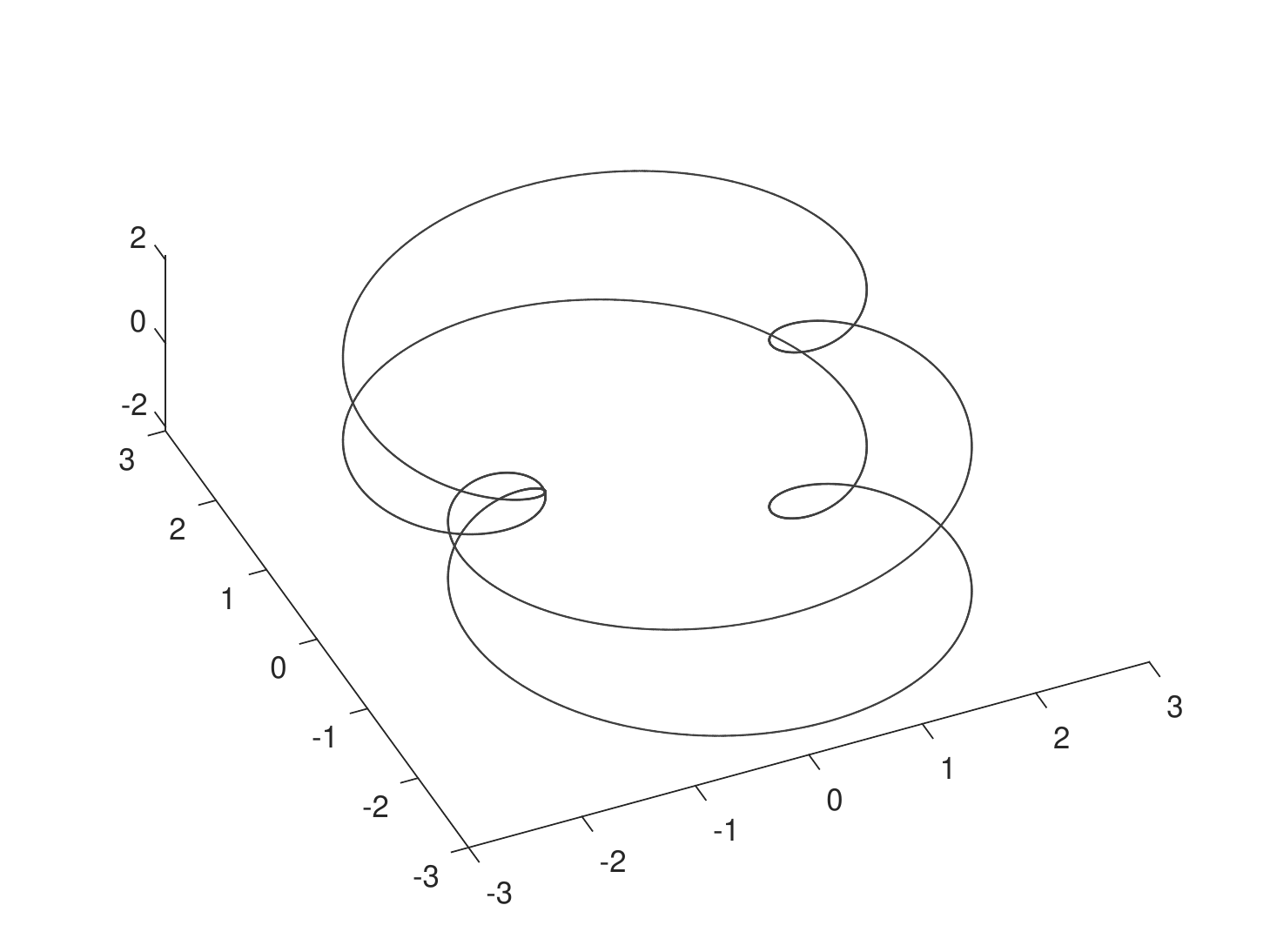}}
	\subfigure[Curve by LSPIA]{
	    \includegraphics[width=0.45\textwidth]{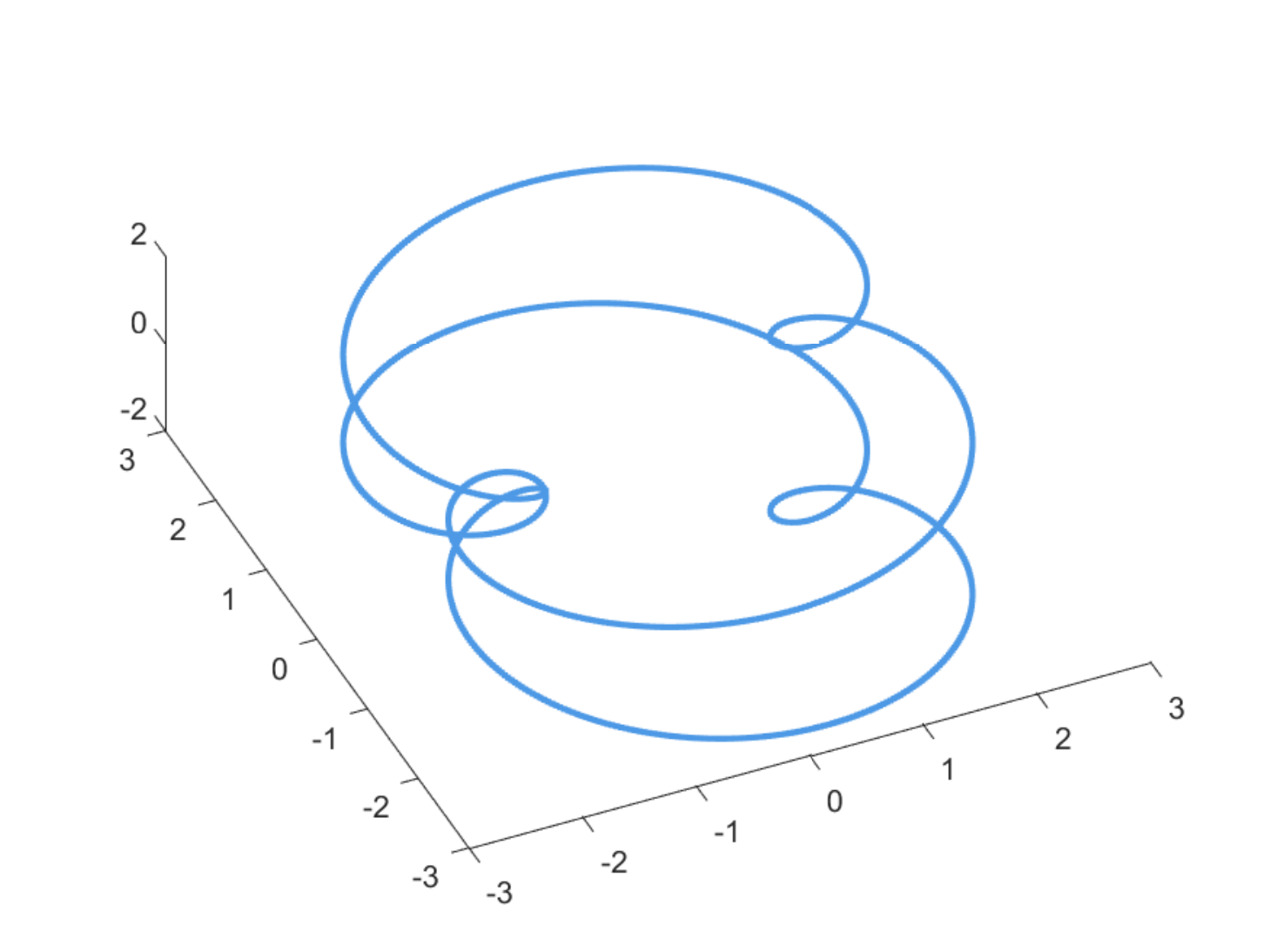}}
	\subfigure[Curve by ALSPIA]{
	    \includegraphics[width=0.45\textwidth]{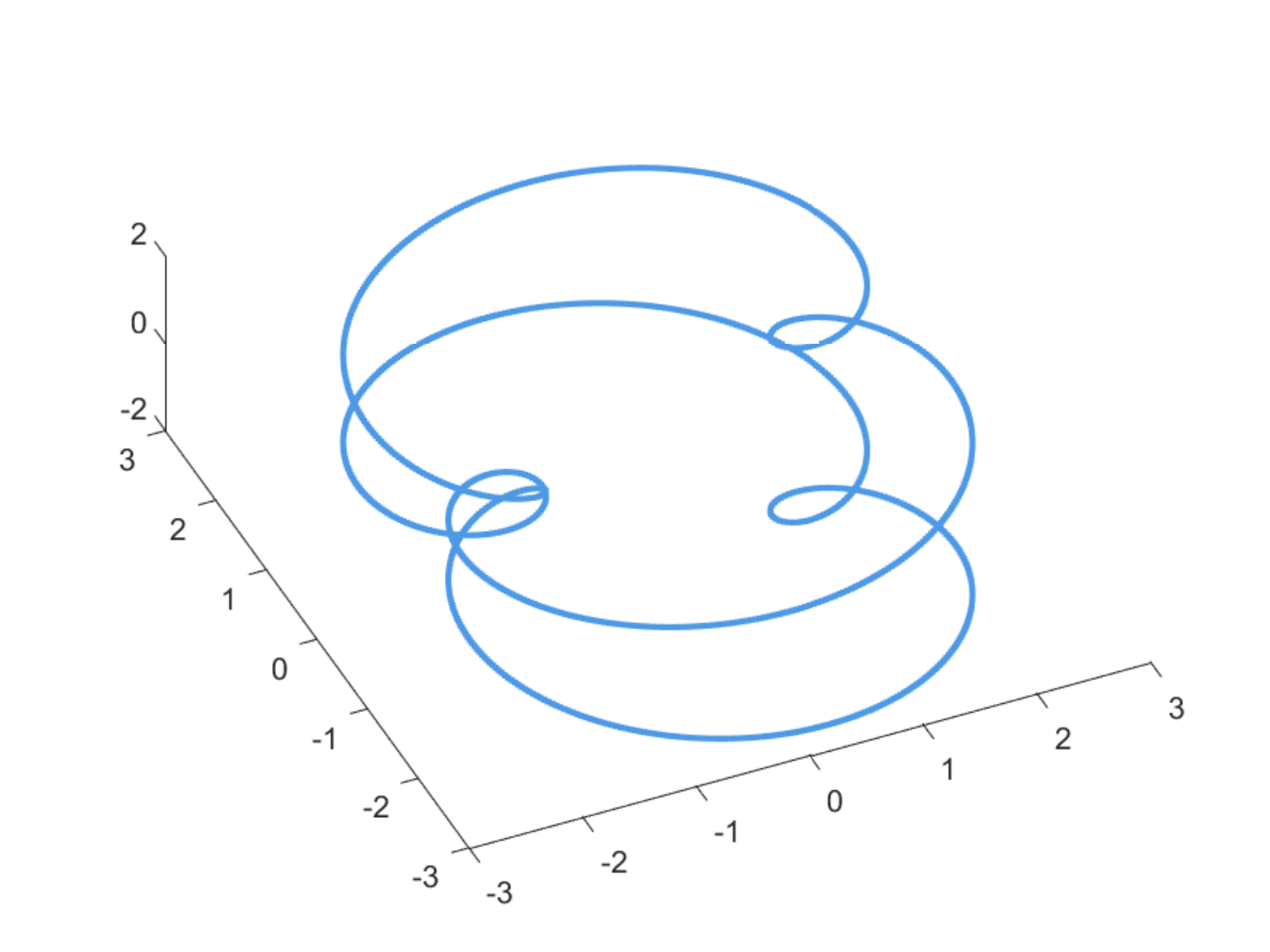}}
    \subfigure[$E_k$ vs CPU]{
	    \includegraphics[width=0.45\textwidth]{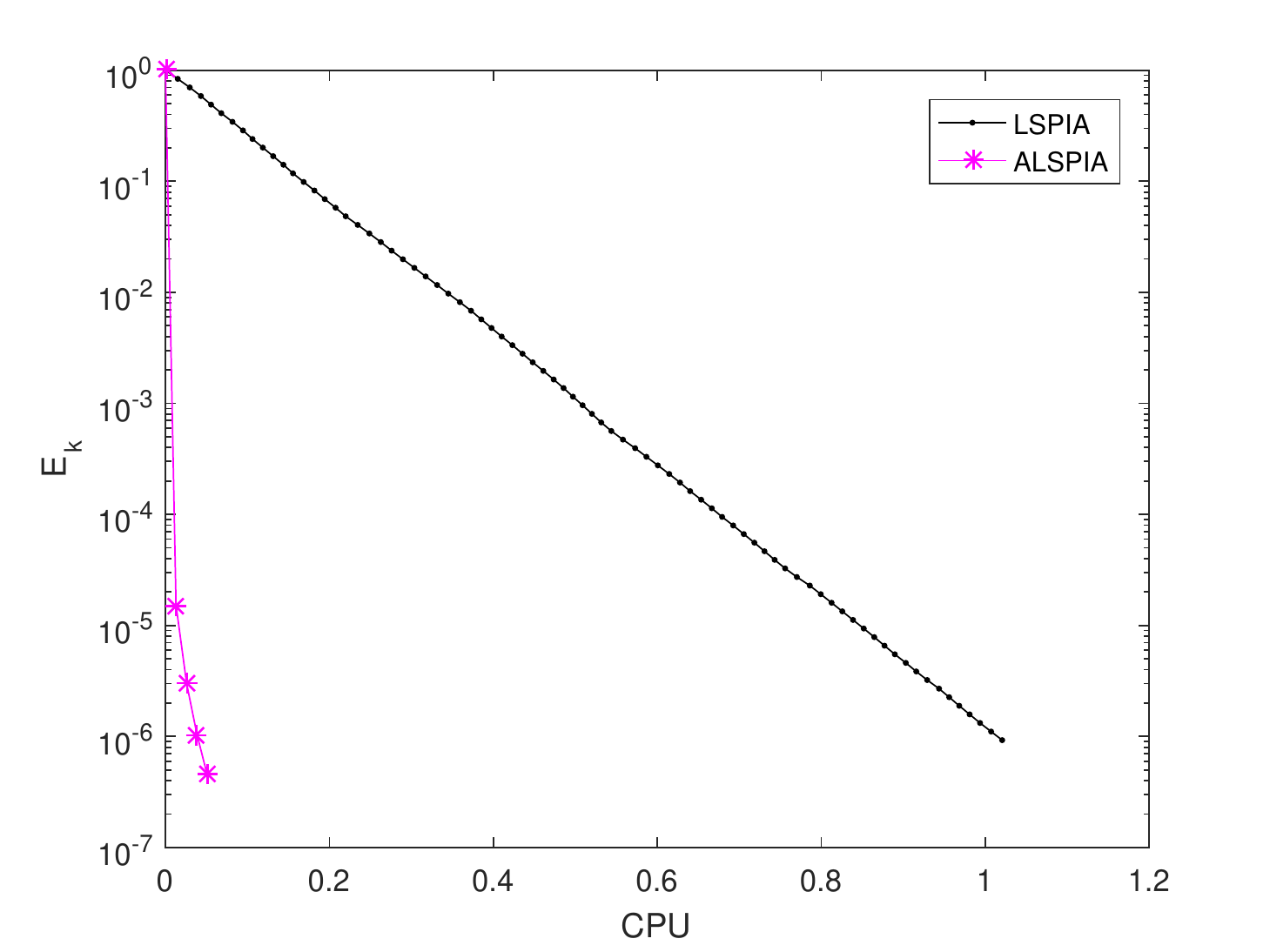}}
\caption{The initial data points, the cubic B-spline fitting curves, and the convergence behaviors of relative fitting error versus CPU time given by LSPIA and ALSPIA with $m=8000$ and $n=1000$ for Example \ref{ex:ALSPIA+curve2}.}
\label{fig:ex2-ALSPIA1+Result}
\end{figure}

(2) \textbf{${\bm A}^T{\bm A}$ is singular.} In the least-squares fitting problems, the rank-deficient collocation matrix may emerge if one takes a  missing data. For more details, we refer to \cite{18LCZ}. In Example \ref{ex:ALSPIA+curve3}, we employ $2001$ control points to fit $14600$ data points with a hole, as shown in Figure \ref{fig:ex_singular} (a). In this case, the rank of ${\bm A}^T{\bm A}$ is $1952$.  We display in Figures \ref{fig:ex_singular} (b) and (c) the cubic B-spline fitting curves generated by singular LSPIA and ALSPIA, respectively. Both of them can fit the initial data points well for this singular least-squares case. However, from Figure \ref{fig:ex_singular} (d), where the iteration history of relative fitting error versus CPU time for the tested methods is given, we can see that the ALSPIA method needs less CPU time than the singular LSPIA method when the relative fitting error is comparable (${\textsc{E}}_\infty = 9.00 \times 10^{-7}$).

For a three-dimensional helix curve in Example \ref{ex:ALSPIA+curve4}, we utilize $3001$ control points to fit $18898$ data points that lack data in three places, which results in the rank of ${\bm A}^T{\bm A}$ being $2847$. The initial data points and the cubic B-spline fitting curves generated by  singular  LSPIA and ALSPIA are shown in Figures \ref{fig:ex_singular2} (a), (b), and (c), respectively. We can see that the given data points are fitted well by ALSPIA in this singular least-squares case. When the stop criterion is achieved, LSPIA and ALSPIA individually require $19$ and $15$ iteration counts and the latter needs less CPU time, as shown in Figure \ref{fig:ex_singular2} (d).  All the shreds of evidence indicate that the ALSPIA method is superior to the singular LSPIA method.

\begin{figure}[!htb]
\centering
    \subfigure[Initial points]{\label{fig:ex_singular:a}
	    \includegraphics[width=0.45\textwidth]{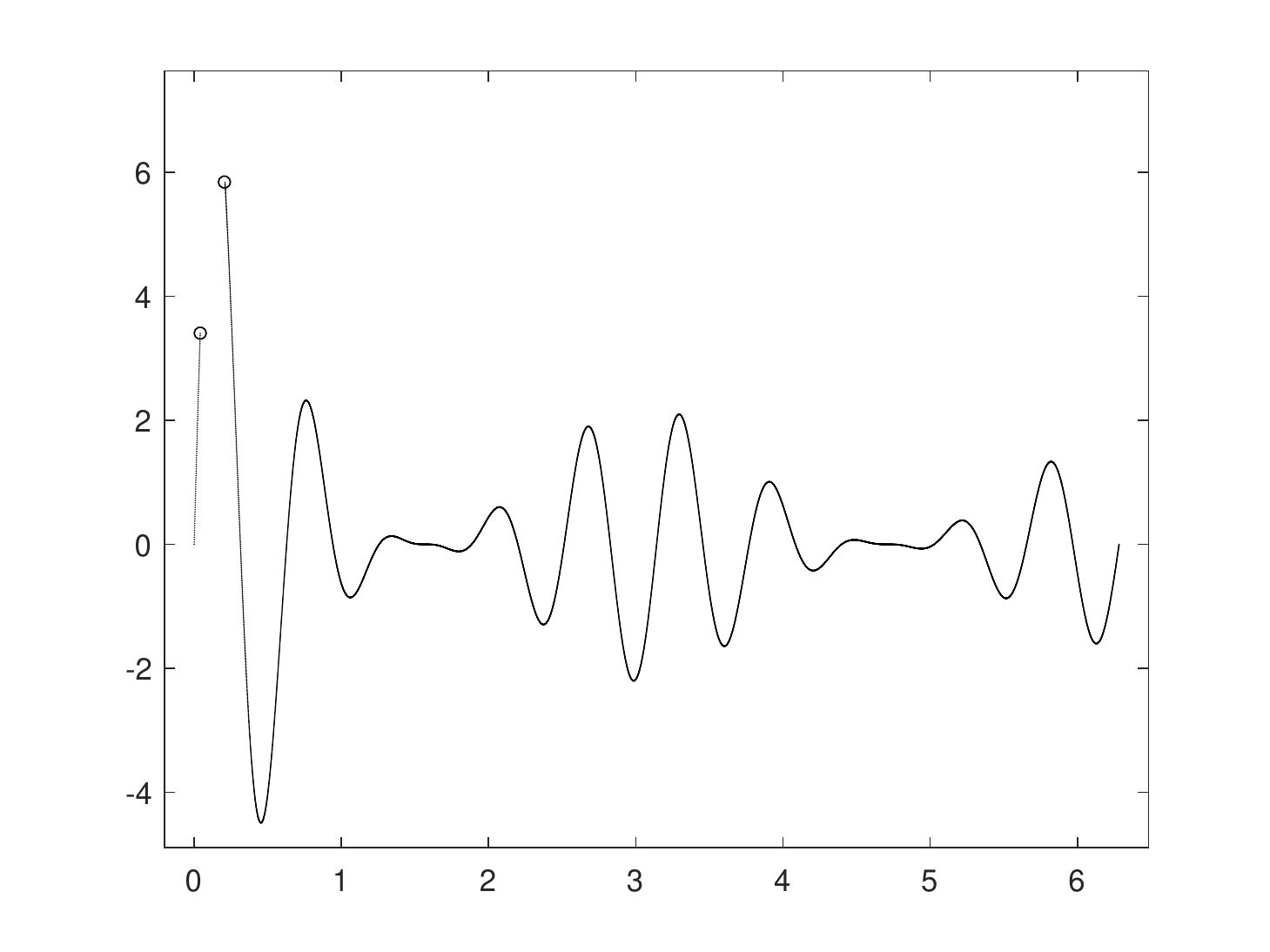}}
	\subfigure[Curve by LSPIA]{
	    \includegraphics[width=0.45\textwidth]{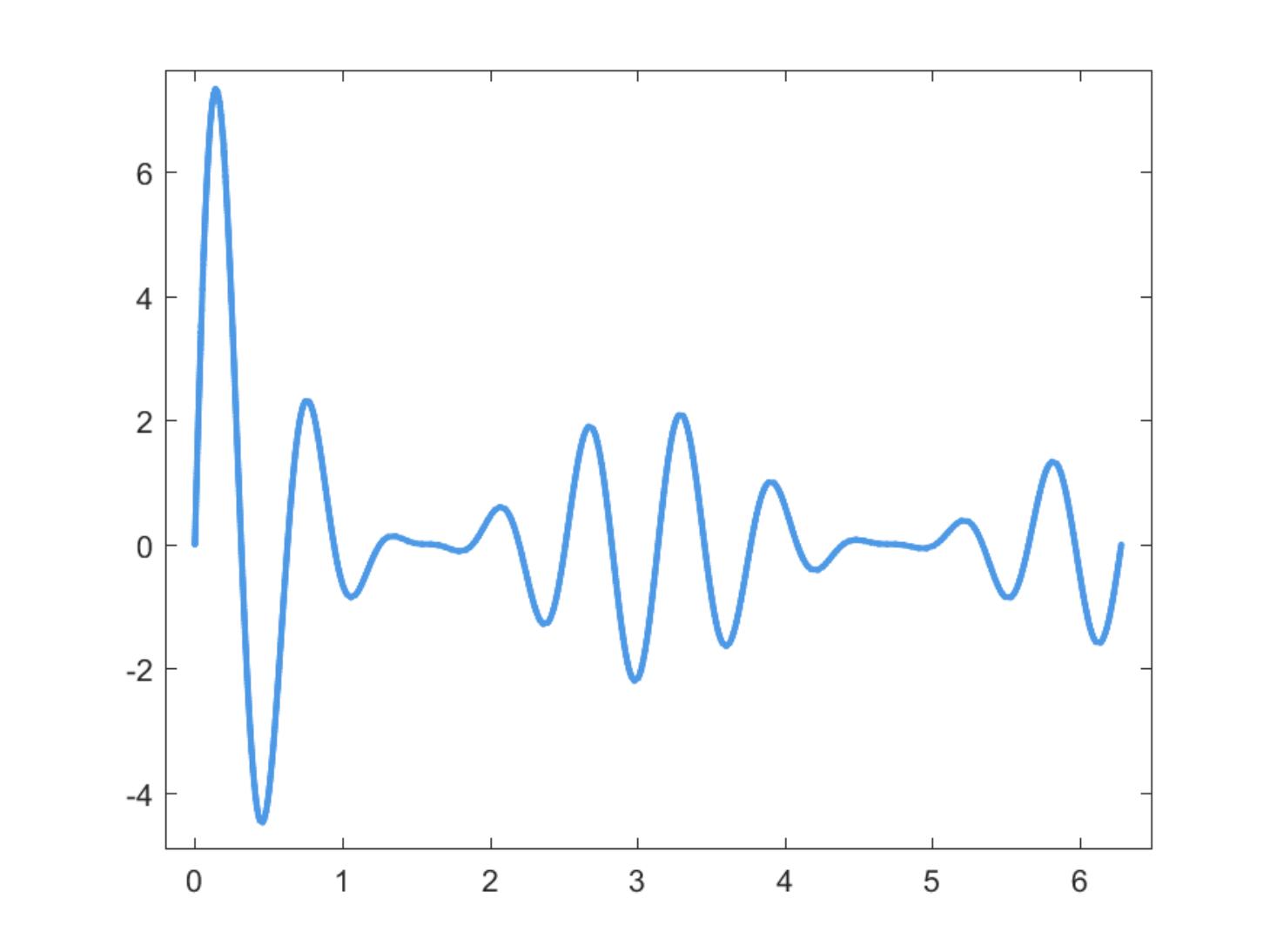}}
	\subfigure[Curve by ALSPIA]{
	    \includegraphics[width=0.45\textwidth]{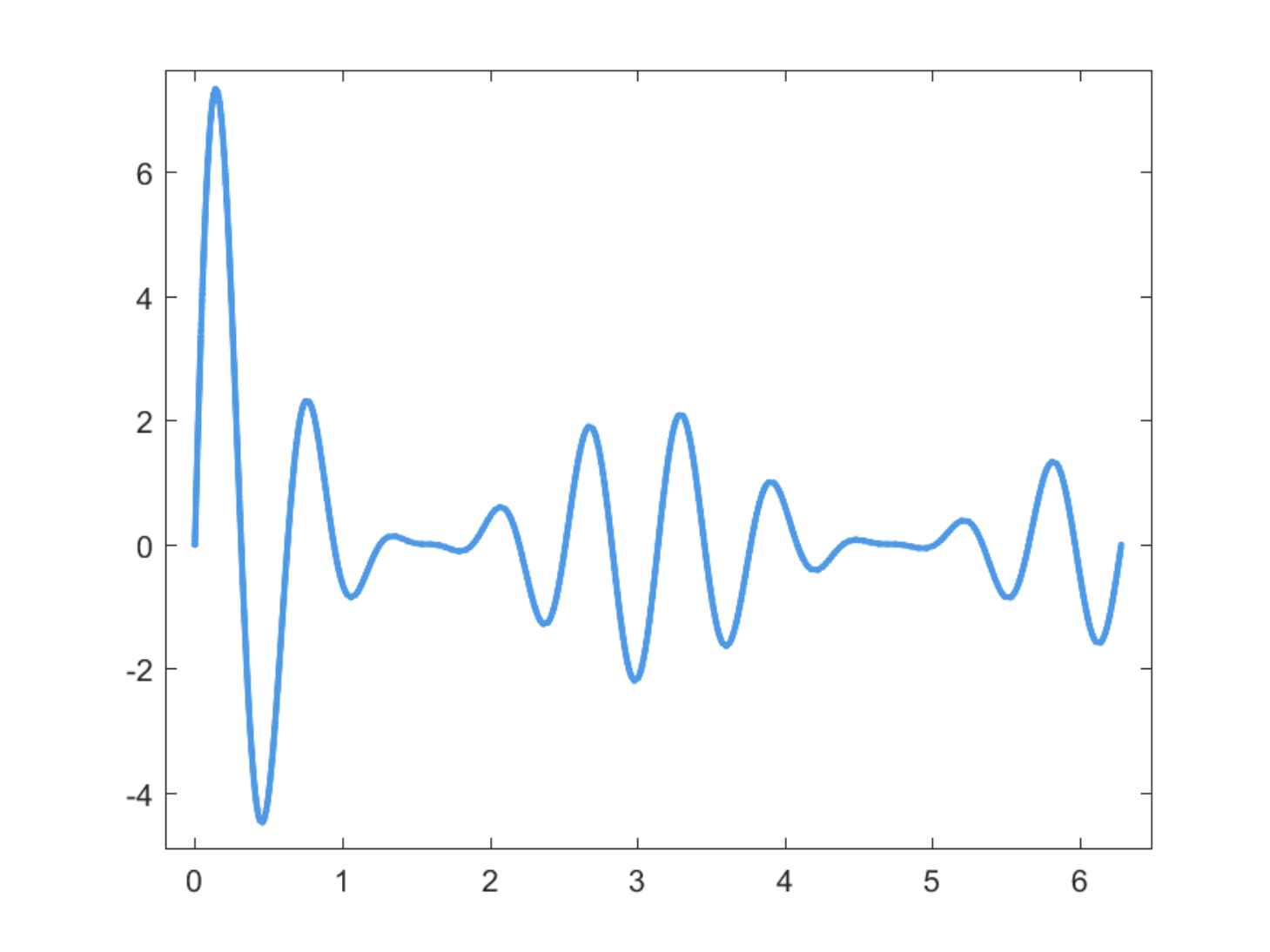}}
    \subfigure[$E_k$ vs CPU]{
	    \includegraphics[width=0.45\textwidth]{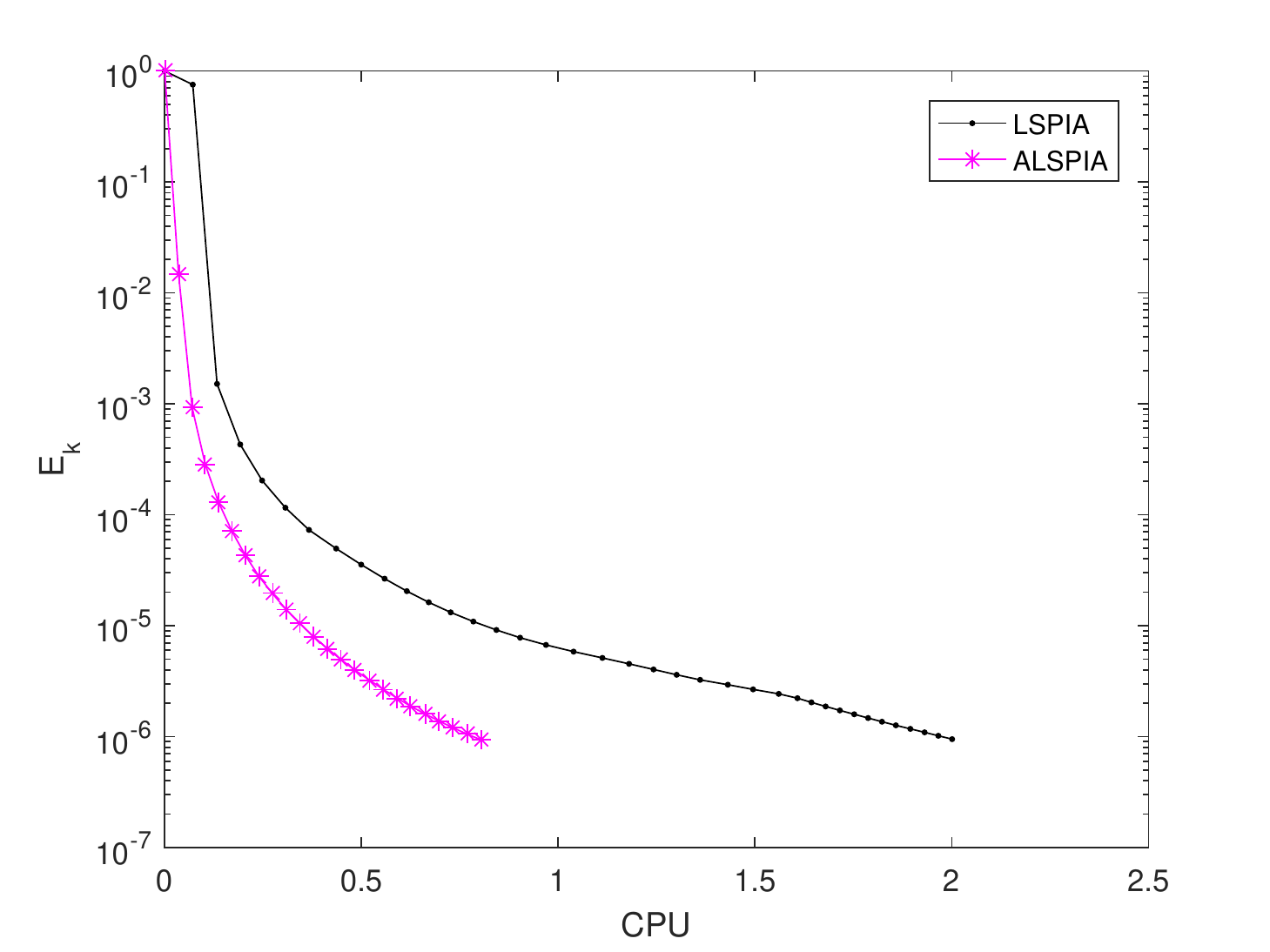}}
\caption{The initial data points, the cubic B-spline fitting curves, and the convergence behaviors of relative fitting error versus CPU time given by LSPIA and ALSPIA with $m=14600$ and $n=2000$ for Example \ref{ex:ALSPIA+curve3}.}
\label{fig:ex_singular}
\end{figure}

\begin{figure}[!htb]
\centering
    \subfigure[Initial points]{
	    \includegraphics[width=0.45\textwidth]{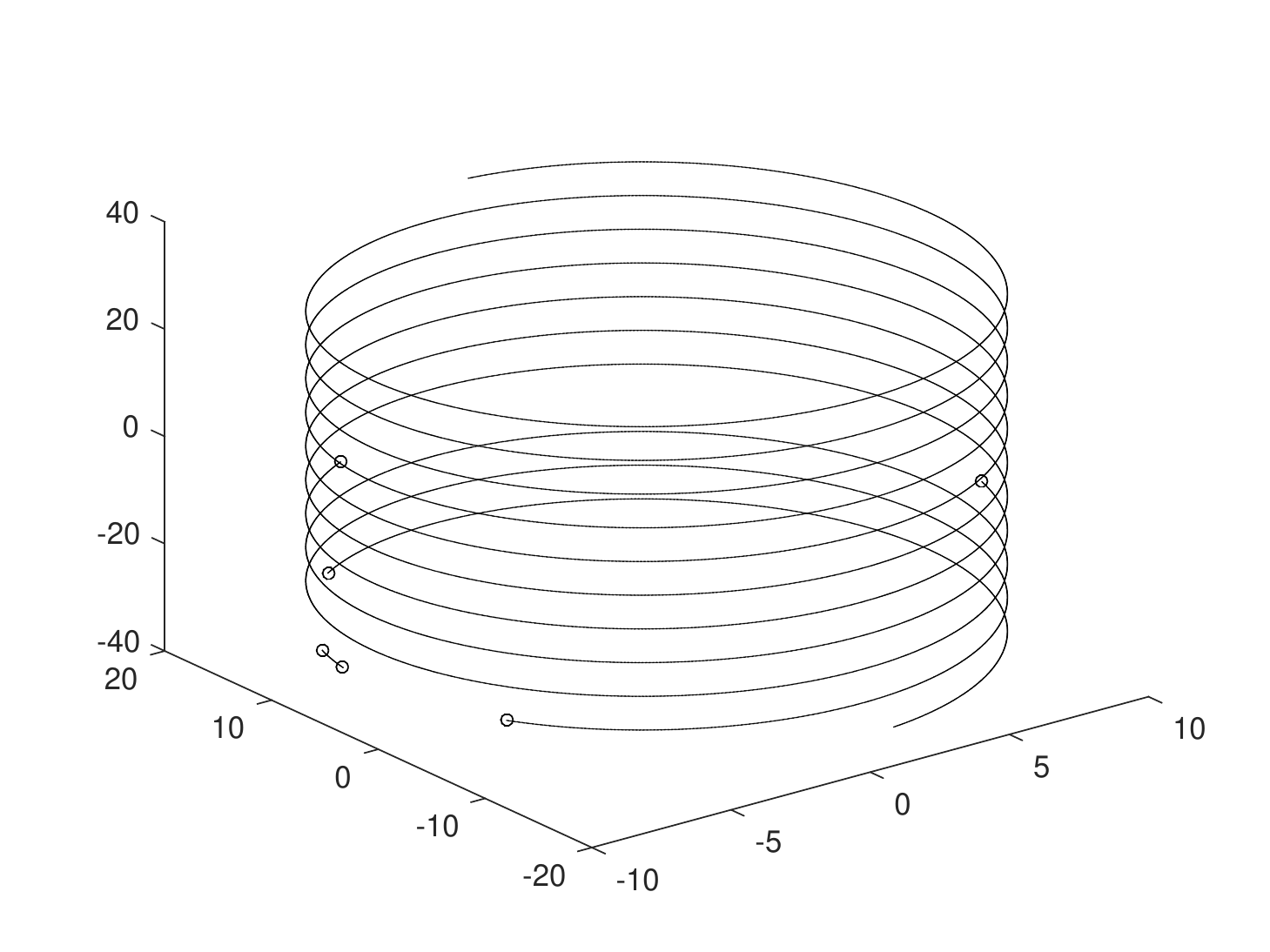}}
	\subfigure[Curve by LSPIA]{
	    \includegraphics[width=0.45\textwidth]{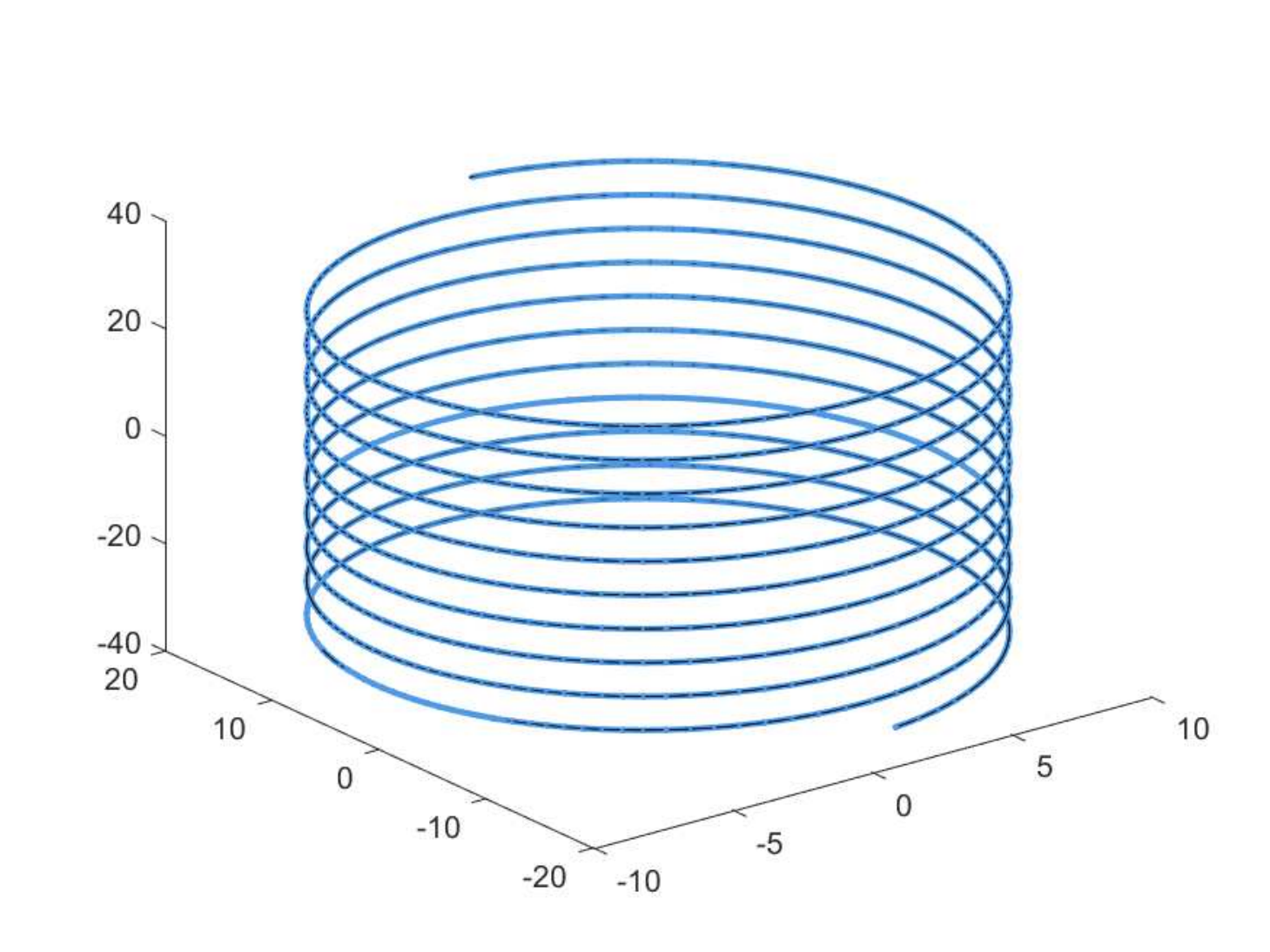}}
	\subfigure[Curve by ALSPIA]{
	    \includegraphics[width=0.45\textwidth]{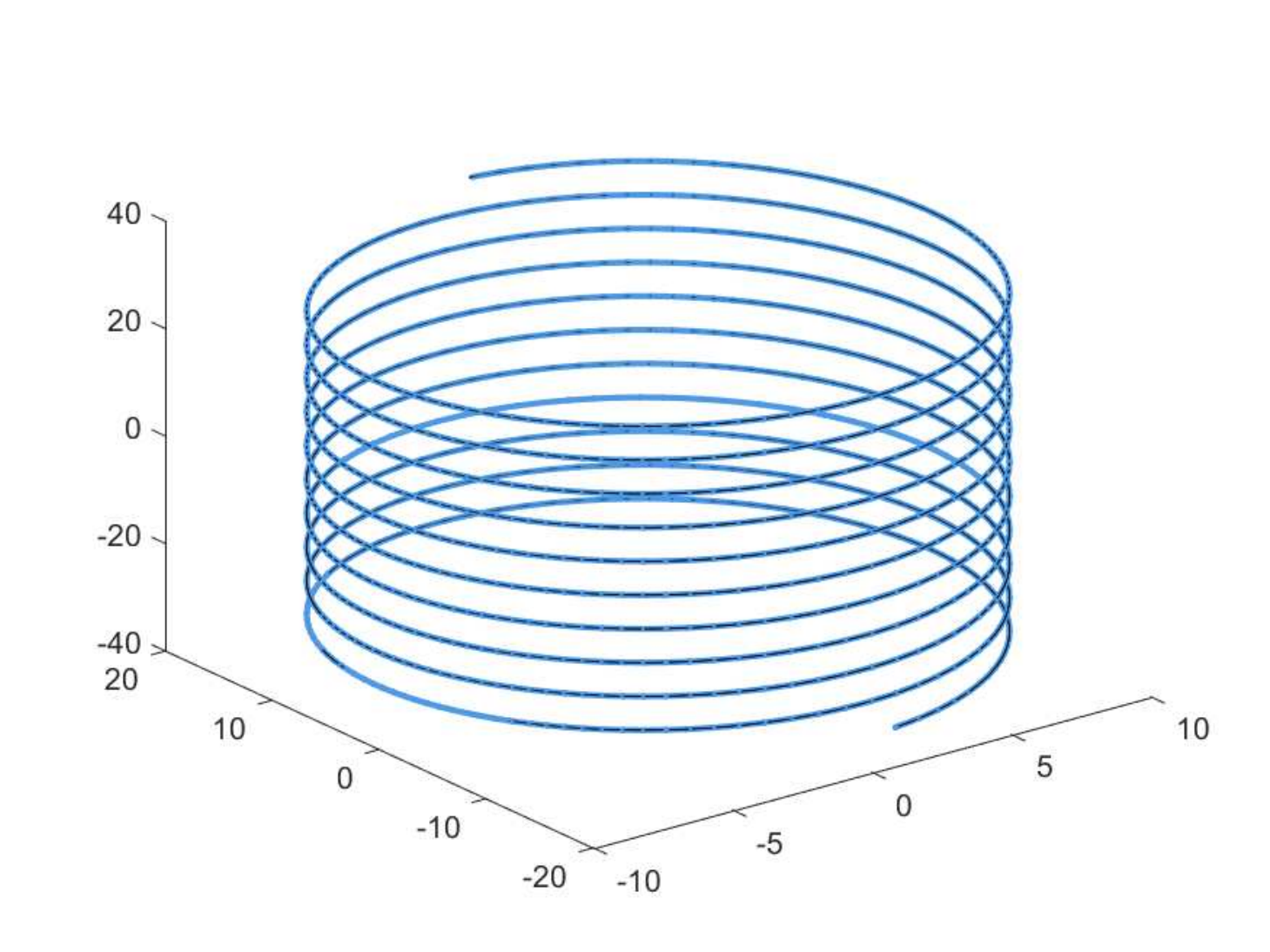}}
    \subfigure[$E_k$ vs CPU]{
	    \includegraphics[width=0.45\textwidth]{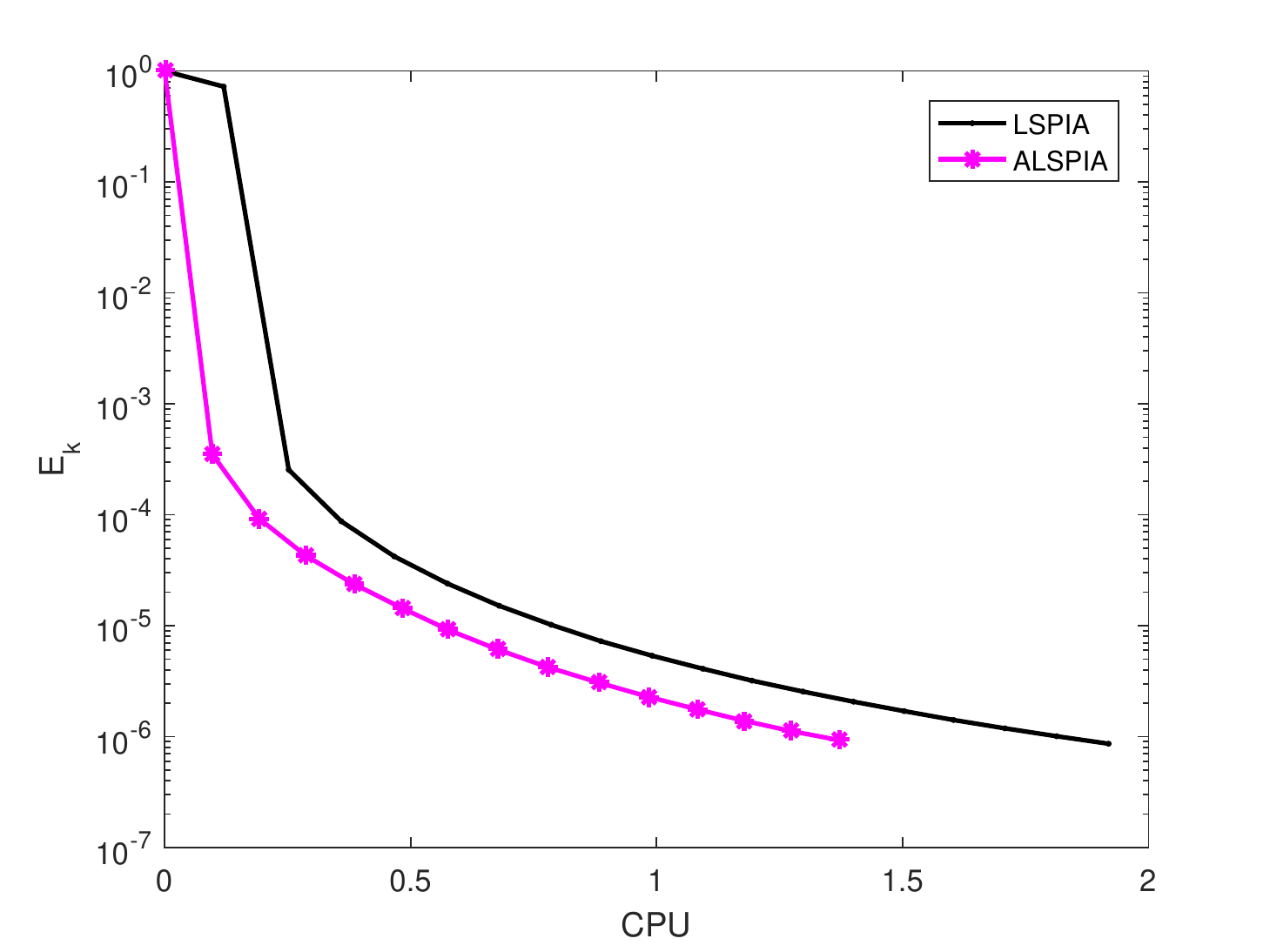}}
\caption{The initial data points, the cubic B-spline fitting curves, and the convergence behaviors of relative fitting error versus CPU time given by LSPIA and ALSPIA with $m=18897$ and $n=3000$ for Example \ref{ex:ALSPIA+curve4}.}
\label{fig:ex_singular2}
\end{figure}

\subsection{Surface fitting}  Similar to the case of curve fitting, we organize the execution details of ALSPIA for surface fitting as follows. Given an ordered point set $\left\{{\bm Q}_{hl} \right\}_{h=0,l=0}^{m,p}$ in $\mathbb{R}^3$, we assign the parameters  $\left\{x_h \right\}_{h=0}^m$ and $\left\{y_l \right\}_{l=0}^p$ as follows.
\begin{equation*}
 x_0 = 0,~x_h = x_{h-1}+\frac{\sum_{t=0}^p \bT{{\bm Q}_{ht}-{\bm Q}_{h-1, t}}}{\sum_{s=1}^m \sum_{t=0}^p \bT{{\bm Q}_{st}-{\bm Q}_{s-1, t}}},
~{\rm and}~
 x_m =1
\end{equation*}
for $h=1,2,\cdots,m$ and
\begin{equation*}
 y_0 = 0,~y_l = y_{l-1}+\frac{\sum_{s=0}^m \bT{{\bm Q}_{sl}-{\bm Q}_{s, l-1}}}{\sum_{h=0}^m \sum_{t=1}^p \bT{{\bm Q}_{st}-{\bm Q}_{s, t-1}}},
~{\rm and}~
 y_p =1
\end{equation*}
for $l=1,2,\cdots,p$. Two knot vectors are defined by
\begin{equation*}
  \left\{
  0,~0,~0,~0,
  ~\bar{x}_4,~\bar{x}_5,\cdots,\bar{x}_n
  ~1,~1,~1,~1
  \right\},
\end{equation*}
  with $\bar{x}_{h+3}=(1-\alpha_x)x_{i-1} + \alpha_x x_i$, $i = \lfloor hd_x \rfloor$, $\alpha_x = hd_x-i$, $d_x = (m+1)/(n-2)$ for $h=1,2,\cdots,n-3$ and
\begin{equation*}
  \left\{
  0,~0,~0,~0,
  ~\bar{y}_4,~\bar{y}_5,\cdots,\bar{y}_n
  ~1,~1,~1,~1
  \right\},
\end{equation*}
  with $\bar{y}_{l+3}=(1-\alpha_y)y_{j-1} + \alpha_y y_j$, $j = \lfloor ld_y \rfloor$, $\alpha_y=ld_y-j$, $d_y = (p+1)/(n-2)$ for $l=1,2,\cdots,n-3$. The initial control points are selected as
\begin{equation*}
{\bm P}^{(0)}_{ij} =  {\bm Q}_{f_1(i), f_2(j)},
\end{equation*}
where $f_1(0)=f_2(0)=0$, $f_1(n)=m$, $f_2(n)=p$, and for $i$, $j\geq 1$, $f_1(i)=\left\lfloor(m+1)i/n \right\rfloor$ and $f_2(j)=\left\lfloor (p+1)j/n \right\rfloor$.

For solving the least-squares problem in Examples  \ref{ex:ALSPIA+surface1} and \ref{ex:ALSPIA+surface2}, we list the numbers of iteration steps, the CPU times, and the relative fitting errors for the LSPIA and ALSPIA methods  with various ($m, p, n$) in Tables \ref{tab:IT+CPU-ex5} and \ref{tab:IT+CPU-ex6}. The results in the two tables show  that ALSPIA can always successfully compute an approximate solution, but LSPIA fails for the cases of ($m, p, n$) being ($50, 50, 20$), ($50, 50, 30$), and ($60, 60, 30$) in Example \ref{ex:ALSPIA+surface1} and  ($50, 50, 20$) in Example \ref{ex:ALSPIA+surface2}, respectively,  due to the number of the iteration steps exceeding $10^{4}$. For all convergent cases, the iteration counts and CPU times of ALSPIA are  appreciably smaller than those of LSPIA, with the IT (resp., CPU) speed-up being at least $30.422$ (resp., $29.535$) and at most attaining even $213.316$ (resp., $226.306$). Hence, the ALSPIA method considerably outperforms the LSPIA method in terms of both iteration counts and CPU times, too. In Figures \ref{fig:ex5-ALSPIA1+Result} and \ref{fig:ex6-ALSPIA1+Result}, we display the initial data points, the bi-cubic B-spline fitting surfaces, and the convergence curves of relative fitting error versus CPU time  given by LSPIA and ALSPIA with a fixed $(m, p, n)=(80, 80, 20)$. Without a doubt, the relative fitting error of ALSPIA is delaying more quickly concerning the increase of the CPU time than that of LSPIA.

\begin{table}[!htb]
 \normalsize
\caption{The numerical results of ${\textsc{E}}_\infty$, IT, and CPU for LSPIA and ALSPIA in Example  \ref{ex:ALSPIA+surface1} with various $m$, $p$, and $n$.}
\centering
\begin{tabular}{ cccccc}
\cline{1-6}
& ($m$, $p$, $n$)        & (50, 50, 20) & (60, 60, 20)& (70, 70, 20)& (80, 80, 20)\\
\cline{1-6}
LSPIA  &{${\textsc{E}}_\infty$}&$\#$&$9.98\times 10^{-7}$&$9.96\times 10^{-7}$&$9.97\times 10^{-7}$\\
       &{IT}        &$>10^{4}$&3561&2984&2676\\
       &{CPU}       &$\#$&7.312&9.346&12.526\\
\cline{1-6}
ALSPIA &{${\textsc{E}}_\infty$}&$9.95\times 10^{-7}$&$9.88\times 10^{-7}$&$9.62\times 10^{-7}$&$9.45\times 10^{-7}$\\
       &{IT}        &65&45&46&38\\
       &{CPU}       &0.092&0.092&0.135&0.168\\
\cline{1-6}
Speed-up & ${\textsc{S}}_{\textsc{it}}$  &$\#$&\textbf{79.133}&\textbf{64.870}&\textbf{70.421}\\
         & ${\textsc{S}}_{\textsc{cpu}}$ &$\#$&\textbf{79.678}&\textbf{69.388}&\textbf{74.452}\\
\hline
\hline
& ($m$, $p$, $n$)        & (50, 50, 30) & (60, 60, 30)& (70, 70, 30)& (80, 80, 30)\\
\cline{1-6}
LSPIA  &{${\textsc{E}}_\infty$}&$\#$&$\#$&$9.99\times 10^{-7}$&$9.99\times 10^{-7}$\\
       &{IT}        &$>10^{4}$&$>10^{4}$&8106&4818\\
       &{CPU}       &$\#$&$\#$&57.960&46.373\\
\cline{1-6}
ALSPIA &{${\textsc{E}}_\infty$}&$9.97\times 10^{-7}$&$9.62\times 10^{-7}$&$9.98\times 10^{-7}$&$9.82\times 10^{-7}$\\
       &{IT}        &41&35&38&39\\
       &{CPU}       &0.122&0.179&0.256&0.368\\
\cline{1-6}
Speed-up & ${\textsc{S}}_{\textsc{it}}$  &$\#$&$\#$&\textbf{213.316}&\textbf{123.538}\\
         & ${\textsc{S}}_{\textsc{cpu}}$ &$\#$&$\#$&\textbf{226.306}&\textbf{125.964}\\
\cline{1-6}
\end{tabular}
\begin{tablenotes}
\footnotesize
\item[1.] {\it The item ' $>10^{4}$' represents that the number of iteration steps exceeds $10^{4}$. In this case, the corresponding relative fitting error, CPU time, and two speed-ups are expressed as $'\#'$.}
\end{tablenotes}
\label{tab:IT+CPU-ex5}
\end{table}

\begin{table}[!htb]
 \normalsize
\caption{The numerical results of ${\textsc{E}}_\infty$, IT, and CPU for LSPIA and ALSPIA in Example  \ref{ex:ALSPIA+surface2} with various $m$, $p$, and $n$.}
\centering
\begin{tabular}{ cccccc}
\cline{1-6}
& ($m$, $p$, $n$)        & (50, 50, 20) & (80, 80, 20)& (100, 100, 20)& (100, 100, 30)\\
\cline{1-6}
LSPIA  &{${\textsc{E}}_\infty$}&$\#$&$9.98\times 10^{-7}$&$9.96\times 10^{-7}$&$9.98\times 10^{-7}$\\
       &{IT}        &$>10^{4}$&3774&1350&1306\\
       &{CPU}       &$\#$&16.552&11.873&24.967\\
\cline{1-6}
ALSPIA &{${\textsc{E}}_\infty$}&$9.94\times 10^{-7}$&$9.99\times 10^{-7}$&$9.58\times 10^{-7}$&$9.24\times 10^{-7}$\\
       &{IT}        &54&47&45&37\\
       &{CPU}       &0.069&0.202&0.385&0.675\\
\cline{1-6}
Speed-up & ${\textsc{S}}_{\textsc{it}}$  &$\#$&\textbf{80.298}&\textbf{30.000}&\textbf{35.297}\\
         & ${\textsc{S}}_{\textsc{cpu}}$ &$\#$&\textbf{81.941}&\textbf{30.858}&\textbf{37.007}\\
\hline
\hline
& ($m$, $p$, $n$)        & (120, 120, 20) & (120, 120, 30)& (120, 120, 40)& (120, 120, 50)\\
\cline{1-6}
LSPIA  &{${\textsc{E}}_\infty$}&$9.97\times 10^{-7}$&$9.96\times 10^{-7}$&$9.99\times 10^{-7}$&$9.98\times 10^{-7}$\\
       &{IT}        &1369&1293&1254&1526\\
       &{CPU}       &19.780&36.243&61.097&114.174\\
\cline{1-6}
ALSPIA &{${\textsc{E}}_\infty$}&$9.34\times 10^{-7}$&$8.83\times 10^{-7}$&$9.02\times 10^{-7}$&$9.47\times 10^{-7}$\\
       &{IT}        &45&31&22&35\\
       &{CPU}       &0.670&0.814&1.068&2.460\\
\cline{1-6}
Speed-up & ${\textsc{S}}_{\textsc{it}}$  &\textbf{30.422}&\textbf{41.710}&\textbf{57.000 }&\textbf{43.600}\\
         & ${\textsc{S}}_{\textsc{cpu}}$ &\textbf{29.535}&\textbf{44.549}&\textbf{57.221 }&\textbf{46.418}\\
\cline{1-6}
\end{tabular}
\label{tab:IT+CPU-ex6}
\end{table}

\begin{figure}[!htb]
\centering
    \subfigure[Initial data points]{
		\includegraphics[width=0.45\textwidth]{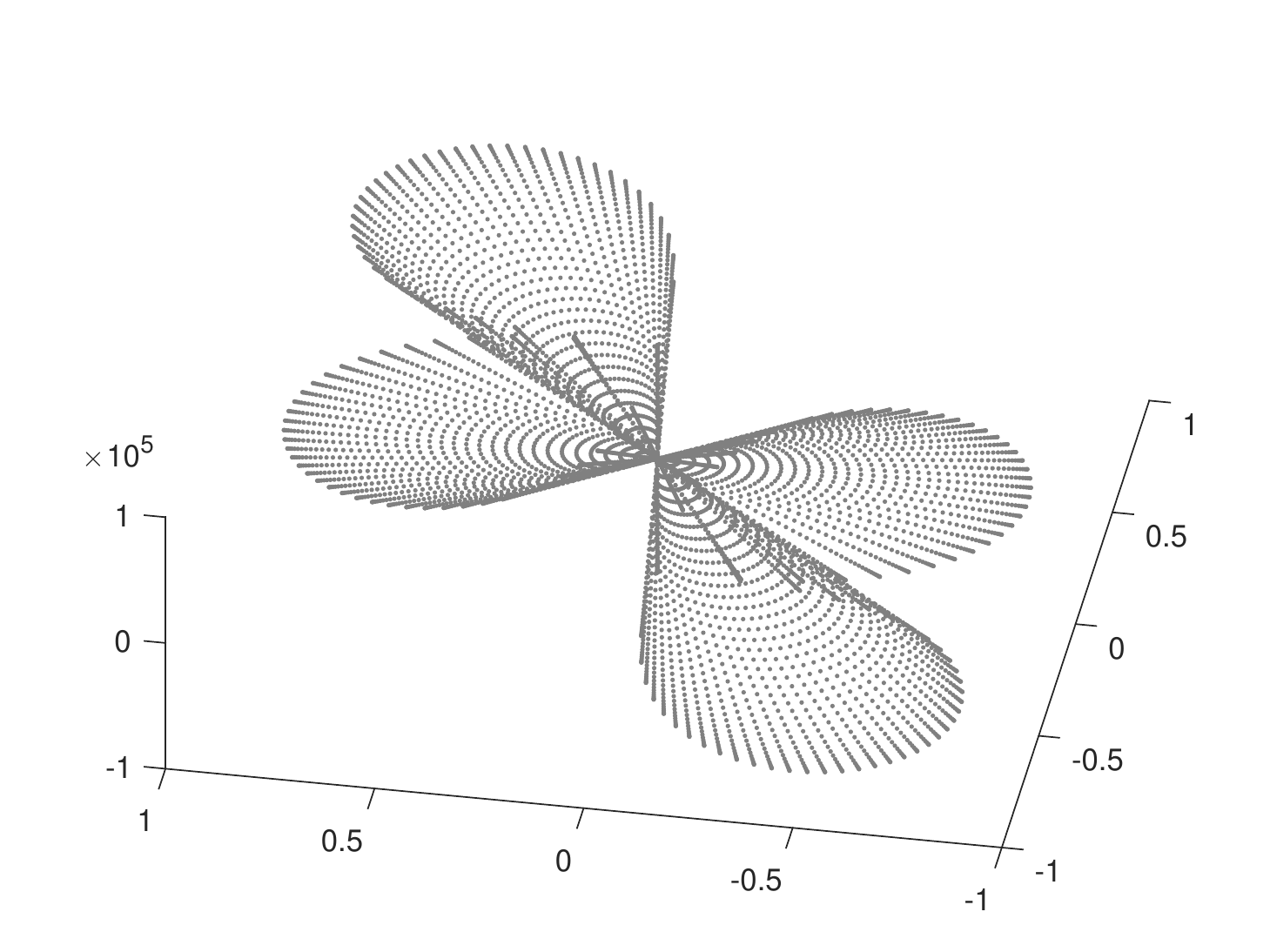}}
	\subfigure[Surface by LSPIA]{
	    \includegraphics[width=0.45\textwidth]{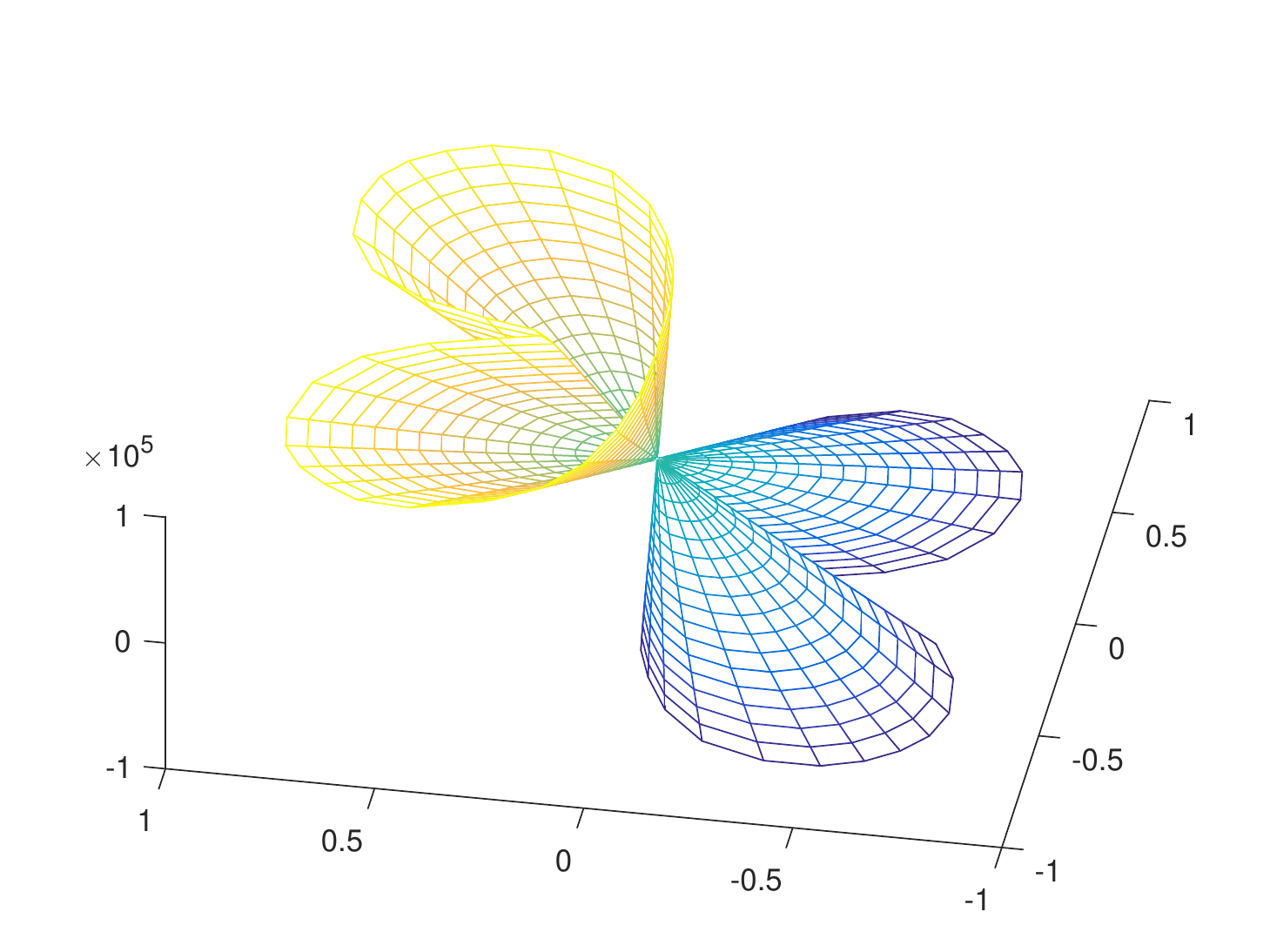}}
	\subfigure[Surface by ALSPIA]{
	    \includegraphics[width=0.45\textwidth]{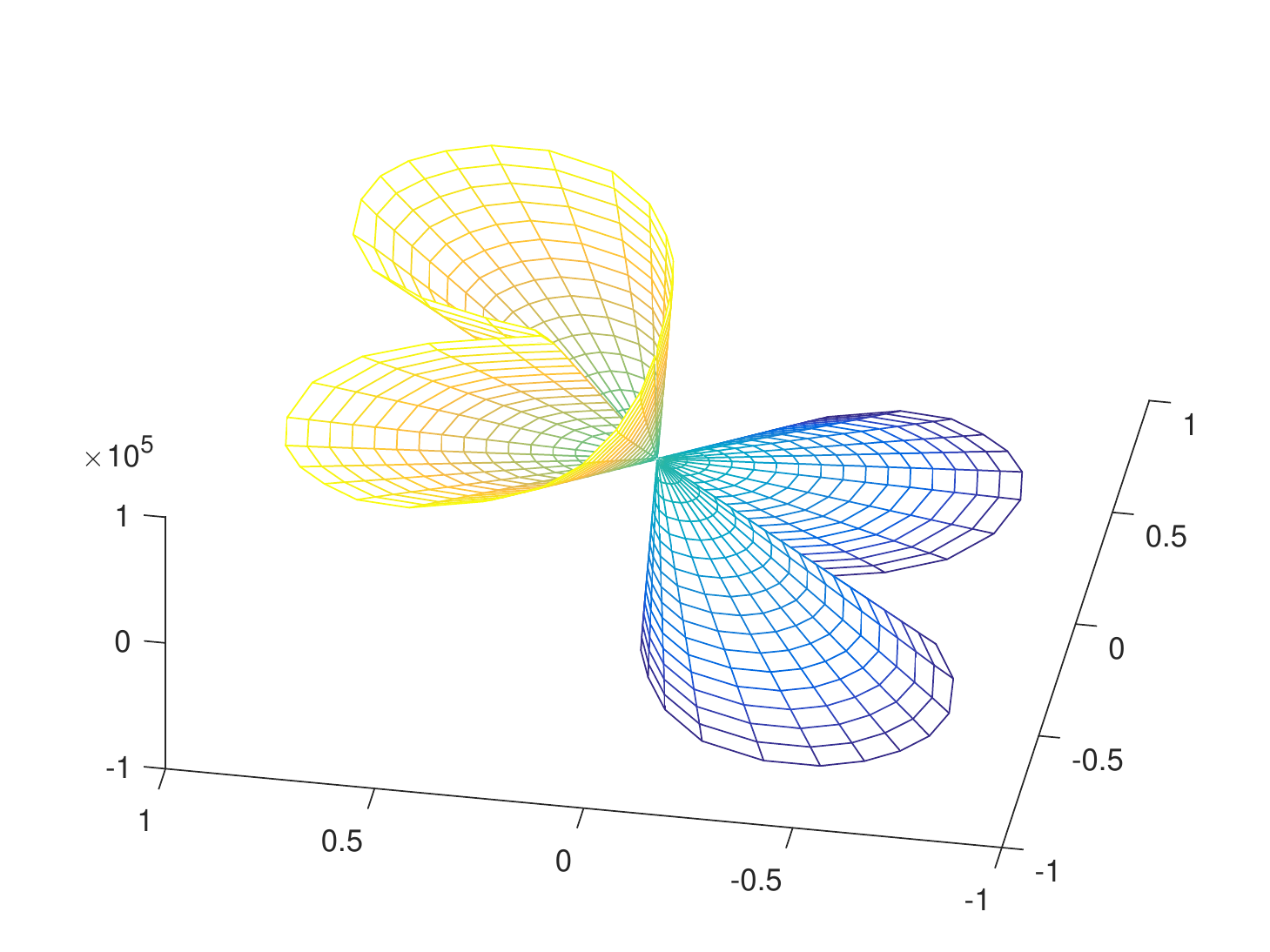}}
	\subfigure[$E_k$ vs CPU]{
	    \includegraphics[width=0.45\textwidth]{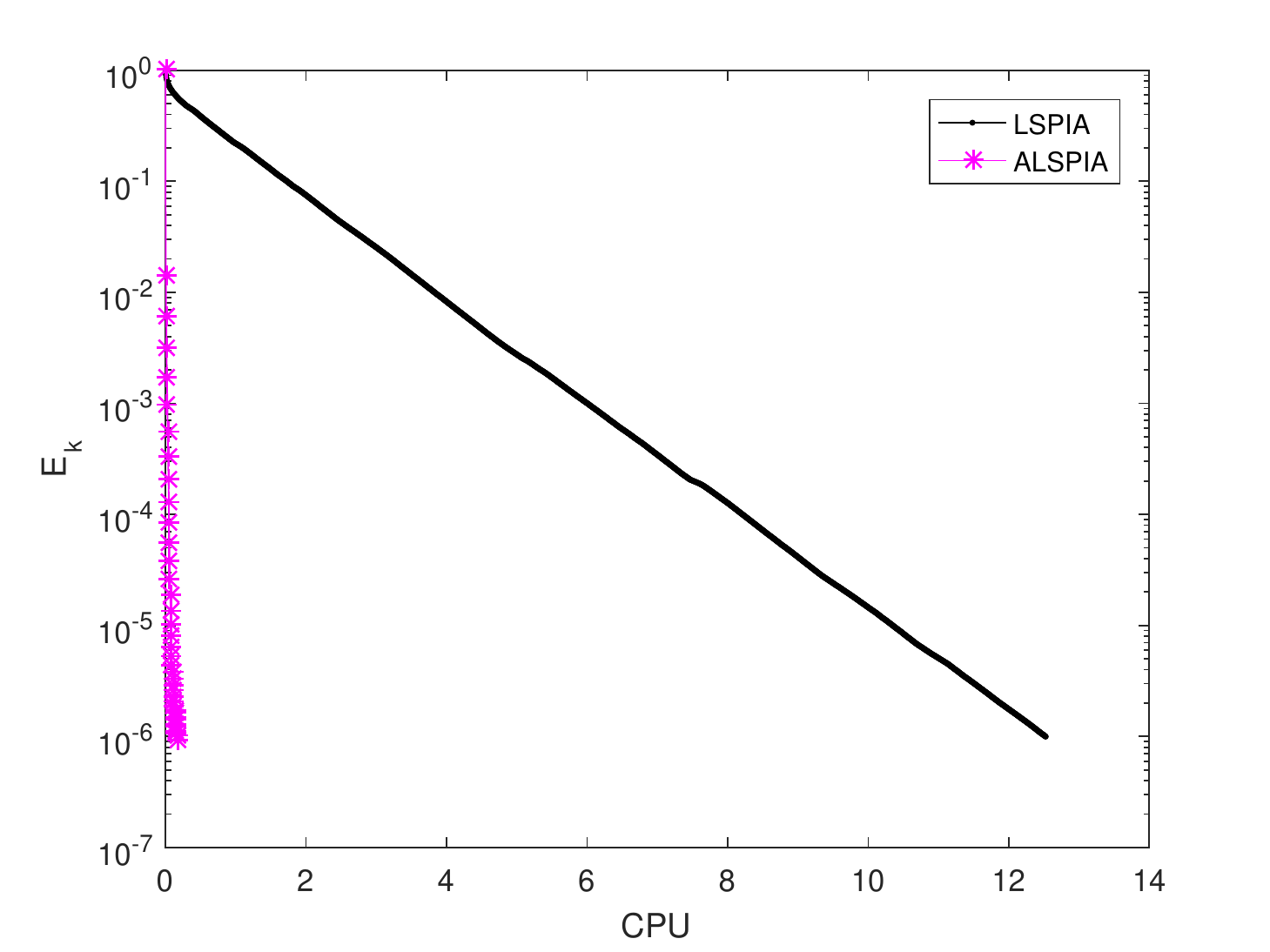}}
\caption{The initial data points, the bi-cubic B-spline fitting surfaces, and the convergence behaviors of relative fitting error versus CPU time given by LSPIA and ALSPIA with $m=80$, $p=80$, and $n=20$ for Example \ref{ex:ALSPIA+surface1}.}
\label{fig:ex5-ALSPIA1+Result}
\end{figure}

\begin{figure}[!htb]
\centering
    \subfigure[Initial data points]{
		\includegraphics[width=0.45\textwidth]{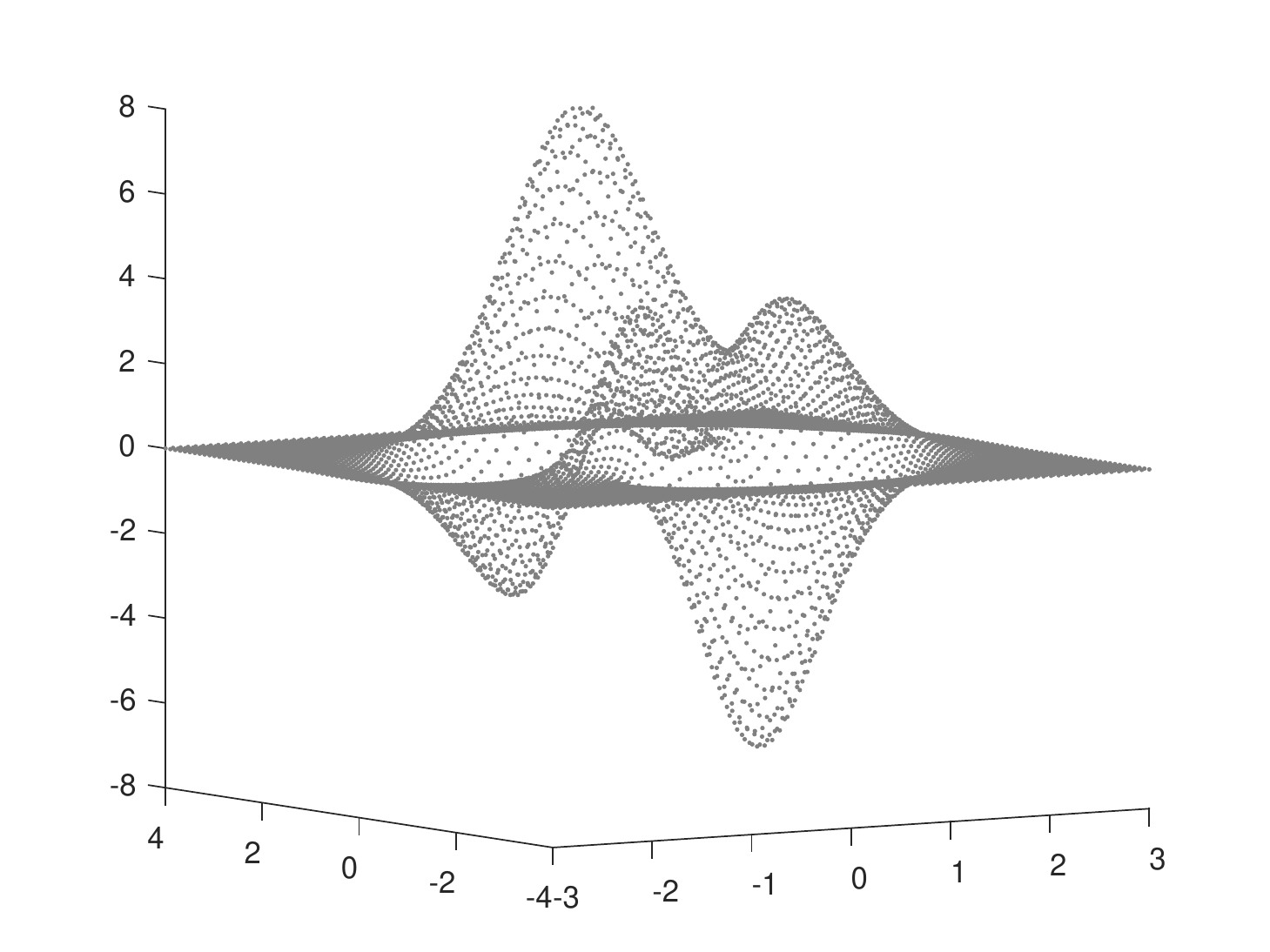}}
	\subfigure[Surface by LSPIA]{
	    \includegraphics[width=0.45\textwidth]{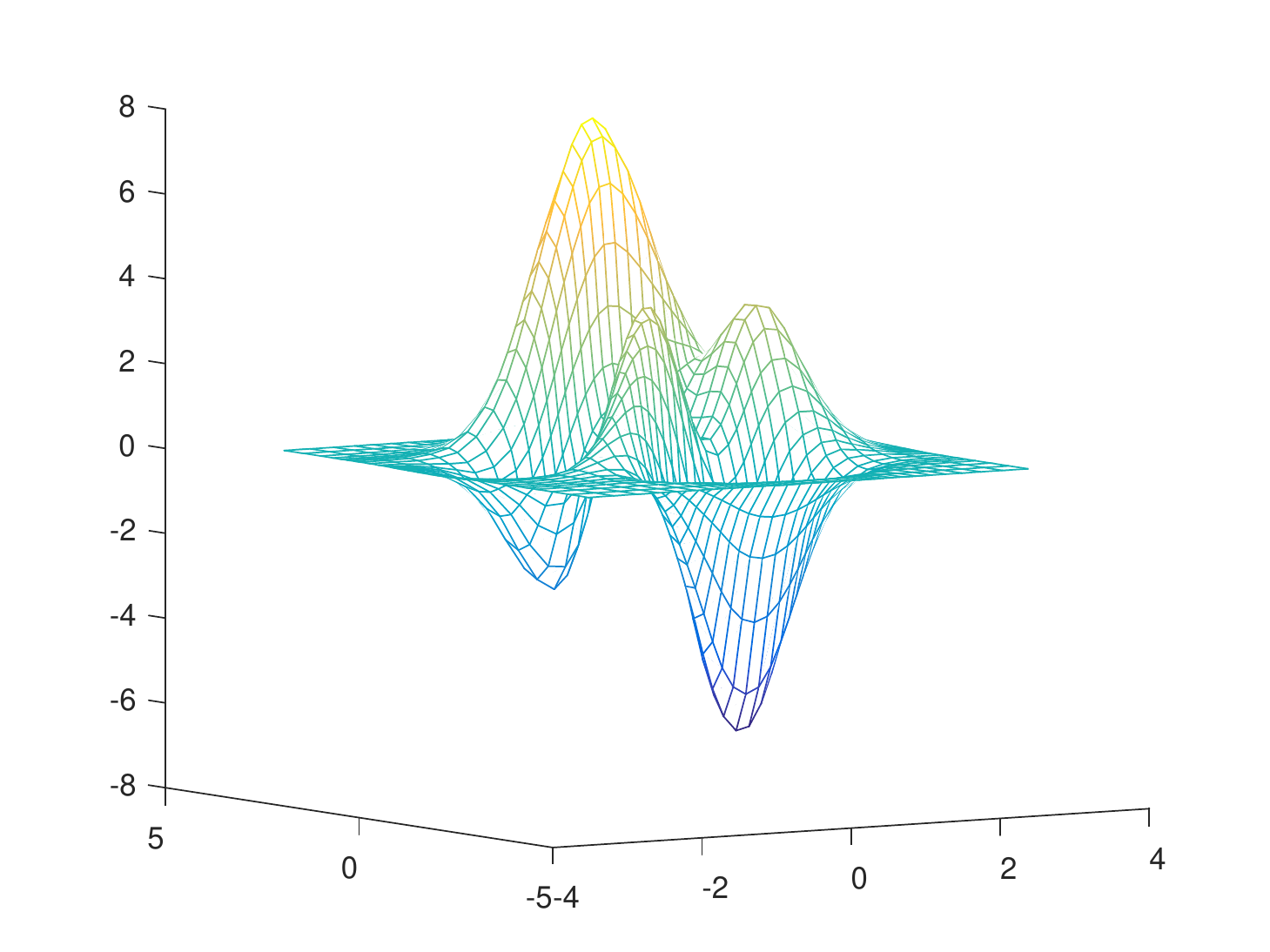}}
	\subfigure[Surface by ALSPIA]{
	    \includegraphics[width=0.45\textwidth]{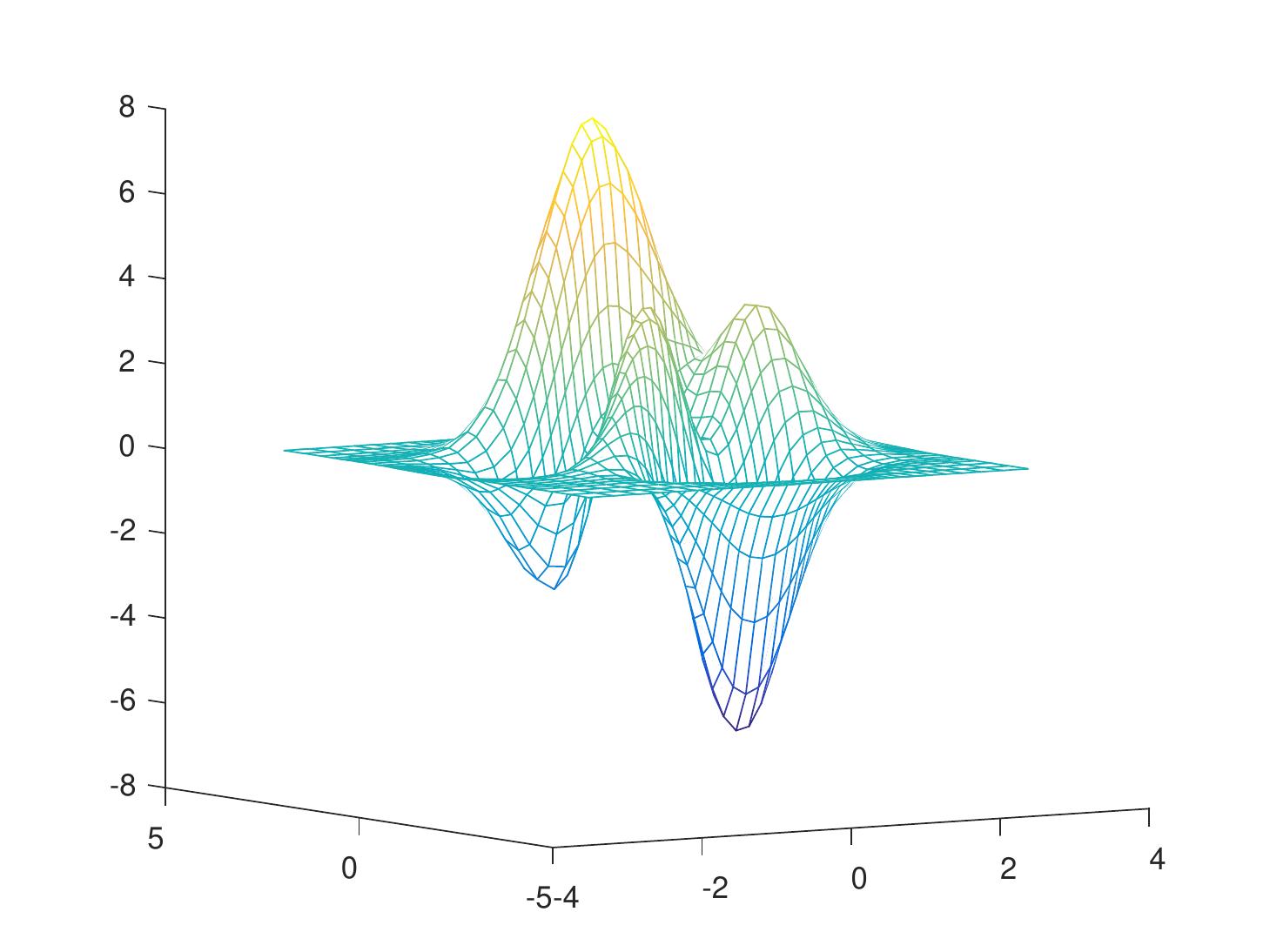}}
	\subfigure[$E_k$ vs CPU]{
	    \includegraphics[width=0.45\textwidth]{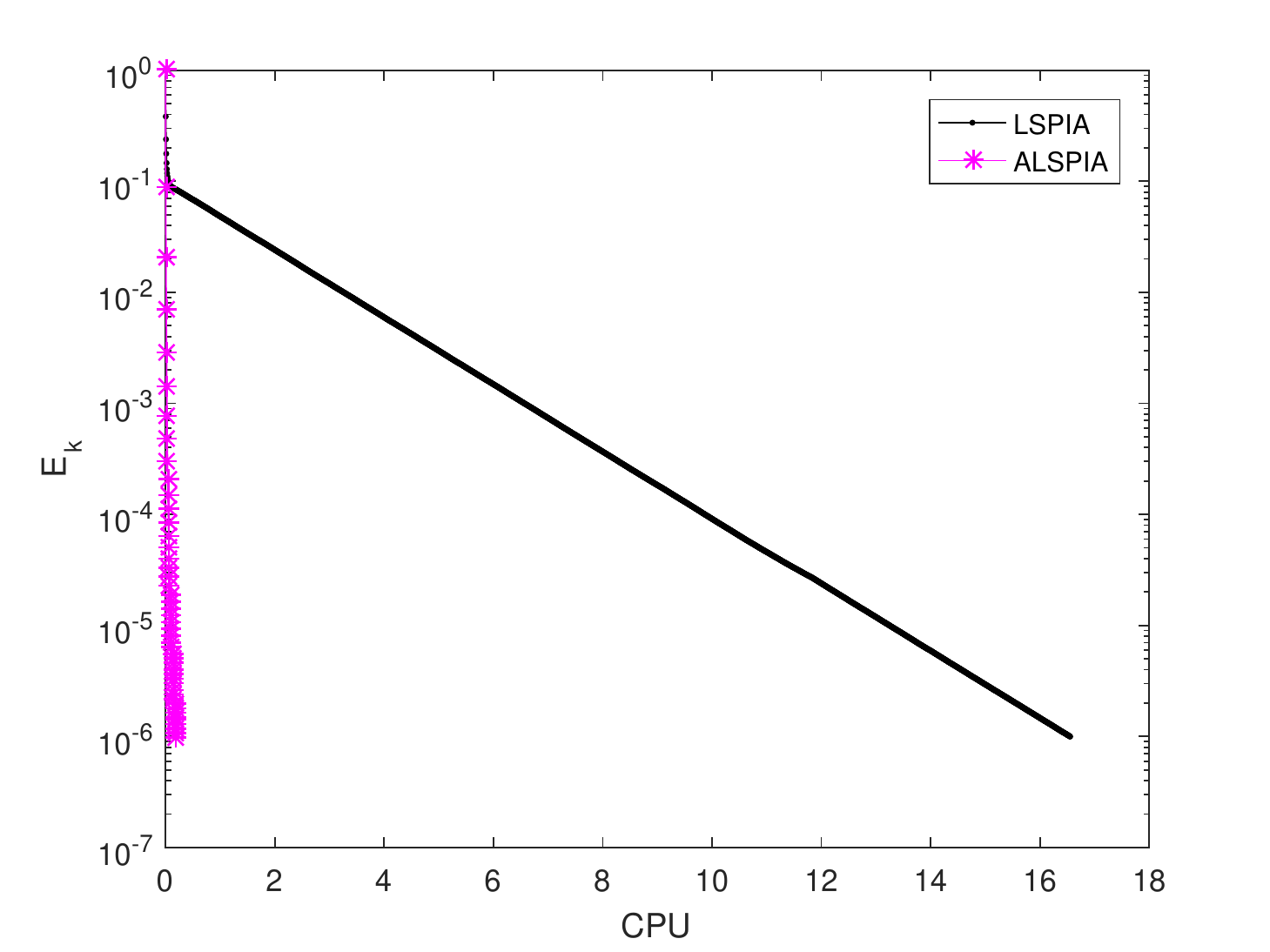}}
\caption{The initial data points, the bi-cubic B-spline fitting surfaces, and the convergence behaviors of relative fitting error versus CPU time given by LSPIA and ALSPIA with $m=80$, $p=80$, and $n=20$ for Example \ref{ex:ALSPIA+surface2}.}
\label{fig:ex6-ALSPIA1+Result}
\end{figure}
\section{Conclusions}\label{Sec:5}
In this work, we propose the ALSPIA method to fit data. Our approach is based on the idea that we update the control points by utilizing an adaptive step size. When the step size is the same constant, the ALSPIA method automatically reduces to the original LSPIA method. We choose them according to the roots of Chebyshev polynomials. We prove that ALSPIA is convergent in two cases. That is, ALSPIA has sub-linear and linear convergence rates when the collocation matrix is rank deficient and of full-column rank, respectively. Our convergence analysis reveals that ALSPIA is faster than the original LSPIA  method. In the numerical experiments, it is shown that ALSPIA  outperforms the classical LSPIA  method in terms of both iteration counts and CPU times.

\section*{Appendix}\label{App}\noindent
In this appendix, we briefly review some properties of the Chebyshev polynomials. For $k\geq 1$, the first kind of Chebyshev polynomials can be defined by the following recursive relation
\begin{align*}
P_0(x)=1,~~P_1(x)=x,~{\rm and}~~  P_{k+1}(x) = 2xP_{k}(x) - P_{k-1}(x).
\end{align*}

Chebyshev polynomials have many interesting properties \cite{97Demmel,Saad2003}. Here are a few, which are easy to prove from their definition.

\begin{lemma}\label{lemma:Chebyshev+Demmel}
\cite[Lemma 6.7]{{97Demmel}} Chebyshev polynomials have the following properties.
\begin{enumerate}[$(1)$]
\setlength{\itemindent}{0.6cm}
\item $P_{k}(1) = 1$.
\item $P_{k}(x) =  \cos(k \cdot\arccos(x))$ if $|x|\leq 1$.
\item $|P_{k}(x)| \leq 1$ if $|x| \leq 1$.
\item The zeros of $P_{k}(x)$ are $x_{\ell} = \cos\left((2\ell+1)\pi/(2k)\right)$ for $\ell=0,1,\cdots,k-1$.
\item $P_{k}(x) = \frac{1}{2} \left[ (x+\sqrt{x^2-1})^k + (x+\sqrt{x^2-1})^{-k} \right] $ if $|x| \geq 1$.
\end{enumerate}
\end{lemma}

An additional theorem that is useful in theoretical work with Krylov subspace methods and that also involves Chebyshev polynomials is as follows.

\begin{lemma}\label{lemma:Chebyshev2}
\cite[Theorem 6.25]{Saad2003} Let $\mathbb{P}_k$ denote the set of all polynomials of degree at most $k$ and $[\alpha,\beta]$ be a real positive interval with $\gamma<\alpha$ or $\gamma>\beta$. The minimum
  \begin{align*}
    \min_{P(x)\in \mathbb{P}_k, P(\gamma)=1}
    \max_{x\in [\alpha,\beta]}
    |P(x)|
  \end{align*}
is reached by the polynomial
\begin{align*}
P^{\ast}_k(x) =
\frac{P_{k}\left( \frac{2x}{\beta-\alpha} -  \frac{\beta+\alpha}{\beta-\alpha} \right)}
     {P_{k}\left( \frac{2\gamma}{\beta-\alpha} -  \frac{\beta+\alpha}{\beta-\alpha} \right)}.
\end{align*}
\end{lemma}

\section*{Acknowledgments}
The authors would like to thank Professor Hongwei Lin very much for his valuable suggestions to construct a rank deficient collocation matrix. This work is partially supported by the Natural Science Foundation of Hunan Province under grant  2020JJ5267 and  the National Natural Science Foundation of China under grant 12101225.



\begin{thebibliography}{99}

\addtolength{\itemsep}{0.5em}

\bibitem{08CLTYC} Zhongxian Chen, Xiaonan Luo, Le Tan, Binghong Ye, and Jiapeng Chen. {\em Progressive interpolation based on Catmull-Clark subdivision surfaces}. Computer Graphics Forum, 2008, 27(7):1823-1827.

\bibitem{97Demmel} James W. Demmel. {\em Applied Numerical Linear Algebra}. SIAM, Philadelphia, PA, USA, 1997.

\bibitem{14DL} Chongyang Deng and Hongwei Lin. {\em Progressive and iterative approximation for least-squares B-spline curve and surface fitting}. Computer-Aided Design, 2014, 47:32-44.

\bibitem{12DM} Chongyang Deng and Weiyin Ma. {\em Weighted progressive interpolation of Loop subdivision surfaces}. Computer-Aided Design, 2012, 44(5):424-431.

\bibitem{19EL} A. Ebrahimi and G. B. Loghmani. {\em A composite iterative procedure with fast convergence rate for the progressive iteration approximation of curves}. Journal of Computational and Applied Mathematics, 2019, 359:1-15.

\bibitem{ENH10} Tommy Elfving, Tourag Nikazad,  and Per Christian Hansen.  {\em  Semi-convergence and relaxation parameters for a class of SIRT algorithms}. Electronic Transactions on Numerical Analysis, 2010, 37:321-336.

\bibitem{Golub2013} Gene H. Golub and Charles F. Van Loan.  {\em Matrix Computations}.  The fourth edition, Johns Hopkins University Press, Baltimore, MD, 2013.

\bibitem{61GV} Gene H. Golub and Richard S. Varga. {\em Chebyshev semi-iterative methods, successive overrelaxation iterative methods, and second order Richardson iterative methods}. Numerische Mathematik, 1961, 3:147-156.

\bibitem{22HL} Yusuf Fatihu Hamza and Hong-Wei Lin. {\em Conjugate-gradient progressive-iterative approximation for Loop and Catmull-Clark subdivision surface interpolation}. Journal of Computer Science and Technology, 2022, 37(2):487-504.

\bibitem{20HW} Zheng-Da Huang and Hui-Di Wang. {\em On a progressive and iterative approximation method with memory for least-square fitting}.  Computer Aided Geometric Design,  2020, 82:101931.



\bibitem{71LF} V. I. Lebedev and S. A. Finogenov. {\em The order of choice of the iteration parameters in the cyclic Chebyshev iteration method}. Computational Mathematics and Mathematical Physics, 1971, 11:425-438.

\bibitem{18LCZ} Hongwei Lin, Qi Cao, and Xiaoting Zhang. {\em The convergence of least-squares progressive iterative approximation for singular least-squares fitting system}. Journal of Systems Science and Complexity, 2018, 31(6):1618-1632.

\bibitem{18LMD} Hongwei Lin, Takashi Maekawa, and Chongyang Deng. {\em Survey on geometric iterative methods and their applications}. Computer-Aided Design, 2018, 95:40-51.

\bibitem{13LZ} Hongwei Lin and Zhiyu Zhang. {\em An efficient method for fitting large data sets using T-splines}. SIAM Journal on Scientific Computing, 2013, 35(6):A3052-A3068.

\bibitem{11LZ} Hongwei Lin and Zhiyu Zhang. {\em An extended iterative format for the progressive-iteration approximation}. Computers and Graphics, 2011, 35(5):967-975.

\bibitem{20LHL} Chengzhi Liu, Xuli Han, and Juncheng Li. {\em Preconditioned progressive iterative approximation for triangular B\'{e}zier patches and its application}. Journal of Computational and Applied Mathematics, 2020, 366:112389.

\bibitem{18LLGZHS} Mingzeng Liu,  Baojun Li, Qingjie Guo, et al. {\em Progressive iterative approximation for regularized least-square bivariate B-spline surface fitting}. Journal of Computational and Applied Mathematics, 2018, 327:175-187.

\bibitem{21LLH} Chengzhi Liu, Zhongyun Liu, and Xuli Han. {\em Preconditioned progressive iterative approximation for tensor product B\'{e}zier patches}. Mathematics and Computers in Simulation, 2021, 185:372-383.

\bibitem{21Wang} Huidi Wang. {\em On extended progressive and iterative approximation for least-squares fitting}. The Visual Computer, 2022, 38:591-602.

\bibitem{21WLLMD} Zhihao Wang, Yajuan Li, Jianzhen Liu, Weiyin Ma, and Chongyang Deng. {\em Gauss-Seidel progressive iterative approximation (GS-PIA) for subdivision surface interpolation}. 2021, https://doi.org/10.1007/s00371-021-02318-9.

\bibitem{02WARV} V. Weiss, L. Andor, G. Renner, and T. V\'{a}radya. {\em Advanced surface fitting techniques}. Computer Aided Geometric Design, 2002, 19(1):19-42.

\bibitem{2014Maxim} Maxim A. Olshanskii and Eugene E. Tyrtyshnikov. {\em Iterative Methods for Linear Systems: Theory and Applications}. SIAM, Philadelphia, PA, USA, 2014.

\bibitem{03PS} V. Pereyr and G. Scherer. {\em Large scale least squares scattered data fitting}. Applied Numerical Mathematics, 2003, 44(1-2):225-239.

\bibitem{97PT} Les Piegl and Wayne Tiller. {\em The NURBS Book}. The second edition, Springer-Verlag, New York, USA, 1997.

\bibitem{22RJ} Dany Rios and Bert J\"{u}ttler. {\em LSPIA, (stochastic) gradient descent, and parameter correction}. Journal of Computational and Applied Mathematics, 2022, 406:113921.

\bibitem{Saad2003} Yousef Saad. {\em Iterative Methods for Sparse Linear Systems}.  The second edition, SIAM, Philadelphia, PA, USA, 2003.

\bibitem{16ZGT} Li Zhang, Xianyu Ge, and Jieqing Tan. {\em Least square geometric iterative fitting method for generalized B-spline curves with two different kinds of weights}. The Visual Computer, 2016, 32:1109-1120.

\end{thebibliography}
\end{document}